\documentclass[a4paper, 11pt]{amsart}
\usepackage{amsmath,amssymb,amscd}
\usepackage[latin1]{inputenc}
\usepackage{graphicx}
\usepackage{enumerate}
\makeatletter
\newcommand{\FM}{{\mathbb F}}

\newcommand{\mt}{\mathbb{T}}
\newcommand{\mr}{\mathbb{R}}
\newcommand{\me}{\mathbb{E}}

\newcommand{\mc}{\mathbb{C}}
\newcommand{\mz}{\mathbb{Z}}
\newcommand{\mh}{\mathbb{H}}

\newcommand{\mn}{\mathbb{N}}

\newcommand{\ms}{\mathbb{S}}

\newcommand{\e}{\bar{e}}
\newcommand{\F}[1]{\ensuremath{\mathbb{F}_{#1}}}

\makeatother
\newtheorem{theorem}{Theorem}[section]

\newtheorem{proposition}[theorem]{Proposition}
\newtheorem{corollary}[theorem]{Corollary}
\newtheorem{lemma}[theorem]{Lemma}
\theoremstyle{definition}
\newtheorem{definition}[theorem]{Definition}
\newtheorem{example}[theorem]{Example}
\newtheorem{remark}[theorem]{Remark}

\usepackage{color}

\title[Distortion and 3-manifolds]{Distortion in graphs of groups and Rapid Decay classification of $3$-manifold groups}
\author{Indira Chatterji, Fran\c{c}ois Gautero}
\date{\today}
\address{Universit\'e C\^ote d'Azur,
CNRS, LJAD (UMR CNRS 7351), Parc Valrose, 06108
Nice Cedex 2, France} \email{Indira.Chatterji@univ-cotedazur.fr, Francois.Gautero@univ-cotedazur.fr}



\begin{document}

\begin{abstract}
The fundamental group of a closed irreducible $3$-dimensional manifold has the Rapid Decay property if and only if it is not virtually Sol. This is proved by studying distortion of length functions in graphs of groups, and the stability of the Rapid Decay property in polynomially distorted graph of groups.
\end{abstract}

\maketitle
\section{Introduction}
A {\em closed $3$-manifold} is a connected, compact $3$-manifold $M$ without boundary. It is {\em irreducible} if every embedded $2$-sphere bounds a $3$-ball. We say that a compact $3$-manifold $M$ {\em has $X$-geometry} if $\pi_1(M)$, the fundamental group of $M$ (we omit the base-point), is a discrete cocompact subgroup of the isometry group of a complete metric space $X$. In that case the group $\pi_1(M)$ has a similar large scale geometry as the metric space $X$ and they are quasi-isometric. Ideas and methods of $3$-dimensional topology have widely spread through geometric group theory. Although a lot is known about their structure, especially since the proof of the Poincare conjecture \cite{Perelman}, they present particularly rich phenomena and form a natural class where to test properties or conjectures, \cite{OyonoPitsch}. 

 The property of {\em Rapid Decay} for finitely generated groups first appeared in a famous paper by Haagerup \cite{Haagerup} (hence also sometimes called ``Haagerup inequality'' \cite{Talbi}), where it was proven for free groups on finitely many letters. The Rapid Decay property found numerous applications for instance in proving Novikov \cite{Connes-Moscovici} and Baum-Connes \cite{VLafforgue} conjectures among others. According to Connes in his seminal book \cite{Cbook}, groups with the Rapid Decay property are the ones admitting a non-commutative dual for which the continuous functions, which correspond to the reduced C*-algebra of the group, admits a subalgebra of smooth functions, which are the functions of rapid decay, namely functions whose decay at infinity is faster than any inverse of a polynomial in a length function over the group.  We refer to \cite{ChatterjiIntro} for a quick introduction and survey about this property, and to Subsection \ref{RD} for the facts relevant to our present discussion. For the purpose of this work, the Rapid Decay property is detected by looking at the algebra $\mc G$ of the group $G$: it is a classical fact that the multiplication of $\mc G$ extends to the $\ell^1$-completion, but not to the $\ell^2$-completion, unless the group $G$ is finite. However, the Rapid Decay property allows a control of the $\ell^2$-norm of the product of two elements in $\mc G$, polynomially in terms of the diameter, measured using a finite generating set, of the support of either one of those elements (Definition \ref{defRD}). The Rapid Decay property is satisfied, among others, by word hyperbolic groups, polynomial growth groups and Jolissaint in \cite{Jolissaint} showed that this property cannot be satisfied by any amenable group with exponential growth, and hence by any group containing an exponentially distorted copy of the integers. Sapir in \cite{Sapir} constructed finitely generated groups without the Rapid Decay property and without amenable subgroups of exponential growth, but it is an open question if such groups can be finitely presented.

We show that containing an exponentially distorted copy of the integers is the only obstruction in the case of 3-manifold groups, and our main result is the following.
\begin{theorem}
\label{sur les 3-varietes} Let $M$ be a closed, irreducible
$3$-manifold. The fundamental group $\pi_1(M)$ has the Rapid Decay property if and only if $M$ does not have Sol-geometry.
\end{theorem}
Recall that the Lie group Sol is defined by the following split extension
$$0 \rightarrow \mr^2 \rightarrow {\rm Sol} \rightarrow \mr \rightarrow
0$$
where $t \in \mr$ acts on $\mr^2$ by $(x,y) \mapsto (e^tx,e^{-t}y)$. Since this is an amenable group with exponential growth, then so is $\pi_1(M)$ for any closed manifold with $\mathrm{Sol}$ geometry and his group does not have Rapid Decay property, so our main theorem is really about the other direction.

Kneser's theorem decomposes any closed orientable $3$-manifold $M$ as the connected sum of {\it prime summands}, that is either manifolds homeomorphic to $\ms^2 \times \ms^1$, or {\it irreducible} ones, namely those for which every $\ms^2$-embedding extends to the ball. Van Kampen's theorem expresses the fundamental group of $M$ as the free product of the fundamental groups of these prime summands. Since $\pi_1(\ms^2 \times \ms^1) = \mz$, it has the Rapid Decay property. Hence by our main theorem, the fundamental group of a prime summand has Rapid Decay property if and only if it does not have $\mathrm{Sol}$-geometry. Moreover, Rapid Decay property is stable by commensurability\footnote{If two groups share an isomorphic finite index subgroup, then one has Rapid Decay property if and only if, the other one does as well.}, so since any non-orientable closed $3$-manifold admits an orientable double cover, we deduce the following
\begin{corollary}
The fundamental group of a closed $3$-manifold has Rapid Decay property if and only if none of the prime summands of an orientable cover has $\mathrm{Sol}$-geometry.
\end{corollary}
Some technical tools for the proof of Theorem \ref{sur les 3-varietes} are of independent interest. One is the following result, which is a generalization of Jolissaint's results \cite{Jolissaint} on stability of the Rapid Decay property.
\begin{proposition}
\label{distortion}
Let $\mathcal G$ be a graph of groups with loose polynomial distortion and finitely generated vertex groups. If each vertex-group has Rapid Decay property, then so does $\pi_1(\mathcal G)$.
\end{proposition}
We refer to Definition \ref{loosepd}, Definition \ref{defpd} and Remark \ref{apdlpd} for the discussion on distortion. The proof of that proposition, as well as basic facts about graphs of groups, their fundamental groups and the Rapid Decay property, is in Section \ref{RDGG}. Section \ref{PolDisto} investigates some technical conditions to ensure (a stronger version of) polynomial distortion for the graphs of groups describing the various manifolds, leading to the following important step.
\begin{proposition}\label{GraphMfldPolDisto}
Let $G$ be the fundamental group of a graph manifold $M$. The graph of groups $\mathcal G$ whose vertex-groups are the fundamental groups of the maximal Seifert submanifolds in $M$, and such that $\pi_1({\mathcal G})=G$, is undistorted.
\end{proposition}
Section \ref{3dmflds} discusses the structure of various manifolds as fundamental groups of graphs of groups having at most polynomial distortion and vertex groups having the Rapid Decay property. We now explain how all the pieces fit together for the proof of our main result.
\begin{proof}[Proof of Theorem \ref{sur les 3-varietes}]
Let $M$ be a closed, irreducible $3$-manifold which does not have Sol-geometry. Because Rapid Decay property is stable up to commensurability, without loss of generality, we can assume that $M$ is orientable and that $\pi_1(M)$ is infinite. If $M$ is finitely covered by a torus bundle over $\ms^1$, then, since by assumption $M$ does not have Sol-geometry, it is finitely covered by the mapping-torus either of a periodic homeomorphism of the torus, or of a reducible one acting as a twist along an invariant curve: in both cases, the fundamental group of such a finite cover is the semi-direct product of $\mz^2$ by $\mz$ acting by an at most polynomially growing automorphism. Hence this group has polynomial growth so by \cite{Jolissaint}[Corollary 2.1.10] has Rapid Decay property, and so does the fundamental group of $M$. Assume then that $M$ is not finitely covered by a torus bundle over $\ms^1$. If $M$ is finitely covered by a torus bundle over the interval, then its fundamental group is commensurable to $\mz^2$, which again has Rapid Decay property because it has polynomial growth.

We are now left with the case where $M$ is a genuine $3$-manifold (Definition \ref{genuine}). If $M$ is a Seifert manifold, then it is virtually a central extension of a hyperbolic group (see Proposition \ref{classic Seifert}) so has Rapid Decay property according to \cite{Nos}. If $M$ is atoroidal and not Seifert, geometrization leaves us with a hyperbolic manifold and Rapid Decay property, both in the closed and finite volume interior (i.e. compact with boundary tori) cases, follows from classical results from the literature, see Proposition \ref{evidence}. Finally assume that $M$ is neither atoroidal nor Seifert, then either $M$ is a graph manifold and Proposition \ref{GraphMfldPolDisto} expresses its fundamental group as the fundamental group of an at most polynomially distorted graph of groups with vertex groups being fundamental groups of Seifert manifolds. So, Proposition \ref{distortion} gives the conclusion. Or, $M$ is a mixed $3$-manifold and again we express $\pi_1(M)$ as the fundamental group of an at most polynomially distorted graph of groups, but now with vertex groups being either fundamental groups of hyperbolic, or Seifert or graph manifolds with boundary, see Lemma \ref{tautologie} and Proposition \ref{mixedprop}. Once again, combined with Proposition \ref{distortion}, this gives the conclusion. \end{proof}
\begin{remark} According to \cite{ChatterjiIntro}, Rapid Decay property for Seifert manifolds and graph manifolds with boundary or for mixed manifolds (with or without boundary) is a consequence of other works like \cite{WangYu}, \cite{GMAVS} or \cite{PrWise}, the last two ones giving a suitable proper action of the fundamental group of a mixed manifold on a finite dimensional CAT(0) cubical complex. In the case of non-Nil manifolds, Theorem \ref{sur les 3-varietes} is also a consequence of the combination theorem for hierarchically hyperbolic spaces given in \cite{BHSII}, see Corollary 7.9 and Theorem 10.1 therein. Notice also that, once proven the Rapid Decay property for atoroidal and graph manifolds, the case of mixed manifolds can be deduced by combining \cite{Dahmani} and \cite{DS}. We do not appeal to any of these works and we treat both closed manifolds and manifolds with boundary without distinction, in a self-contained way. Using the geometry in the decomposition as graphs of groups to understand the large scale structure of 3-manifolds groups has been recently used in \cite{HRSS} or \cite{Zbinden} for example. 
\end{remark}
{\bf Acknowledgements:} We thank Jason Behrstock for the discussion on hierarchically hyperbolic spaces as well as the related references, and Luisa Paoluzzi for her interest. I. Chatterji is partially supported by the ANR projects GALS ANR-23-CE40-0001 and GOFR ANR-22-CE40-0004.
\section{Graphs of groups and the Rapid Decay property}
\label{RDGG}
In this section we recall basic facts about graphs of groups, show that polynomial distortion, which is a geometric condition on the graph of groups, allows to deduce the Rapid Decay property. Recall that a {\em length function} on a group $G$ is a positive real function which is symmetric, sub-additive with respect to the group operation and vanishes on the neutral element. For instance, the word-length associated to any generating set of $G$ is a length function, but we will only be dealing with finite generating sets here. If $G$ embeds in a larger finitely generated group $H$, then $G$ also inherits a length function from the length from $H$, but those could be quite different, and this difference is crucial in what follows.
\subsection{The fundamental group of a graph of groups} We borrow the following presentation from \cite{ReinfeldtWeidmann} (see also \cite{KapovichWeidmann}), and refer the interested reader to \cite{Serre} for an introduction to Bass-Serre theory (another point of view is developed in \cite{Dicks}).
\begin{definition}\label{gog}
A {\it graph of groups} is a quadruplet ${\mathcal G}=(\Gamma,\{G_e\},\{G_v\},\{\imath_e\})$ where:
\begin{itemize}
  \item $\Gamma=(V,E)$ is a graph, not necessarily simplicial,
  \item for each oriented edge $e\in E$, the group $G_e$ associated to the edge $e$, called {\em edge-group} is finitely generated
 and $G_e =G_{\e}$, where $\e$ is the edge $e$ with the opposite orientation,
 \item for each vertex $v\in V$, the {\em vertex-group} group $G_v$ is discrete and finitely generated,
 \item  for each oriented edge $e\in E$, there is a monomorphism $\imath_{e} \colon G_e \rightarrow
G_{t(e)}$, where $t(e)$ is the terminal vertex of $e$. 
\end{itemize}
Since the edge $e$ is oriented, there is also a monomorphism $\imath_{\e} \colon G_e \rightarrow
G_{i(e)}$, where $i(e)$ is the initial vertex of $e$. The {\it crossing edge} maps $\imath_{\e}(G_e)<G_{i(e)}$, the image of the edge-group in its initial vertex, to the terminal vertex-group $G_{t(e)}$. It spells out as
$$c_e=\imath_e\circ\imath_{\e}^{-1}:\imath_{\e}(G_e)<G_{i(e)}\to G_{t(e)},$$
and $c_e\circ c_{\e}=id$.
\begin{figure}[htpb]
    \centering
        \includegraphics[width=\linewidth]{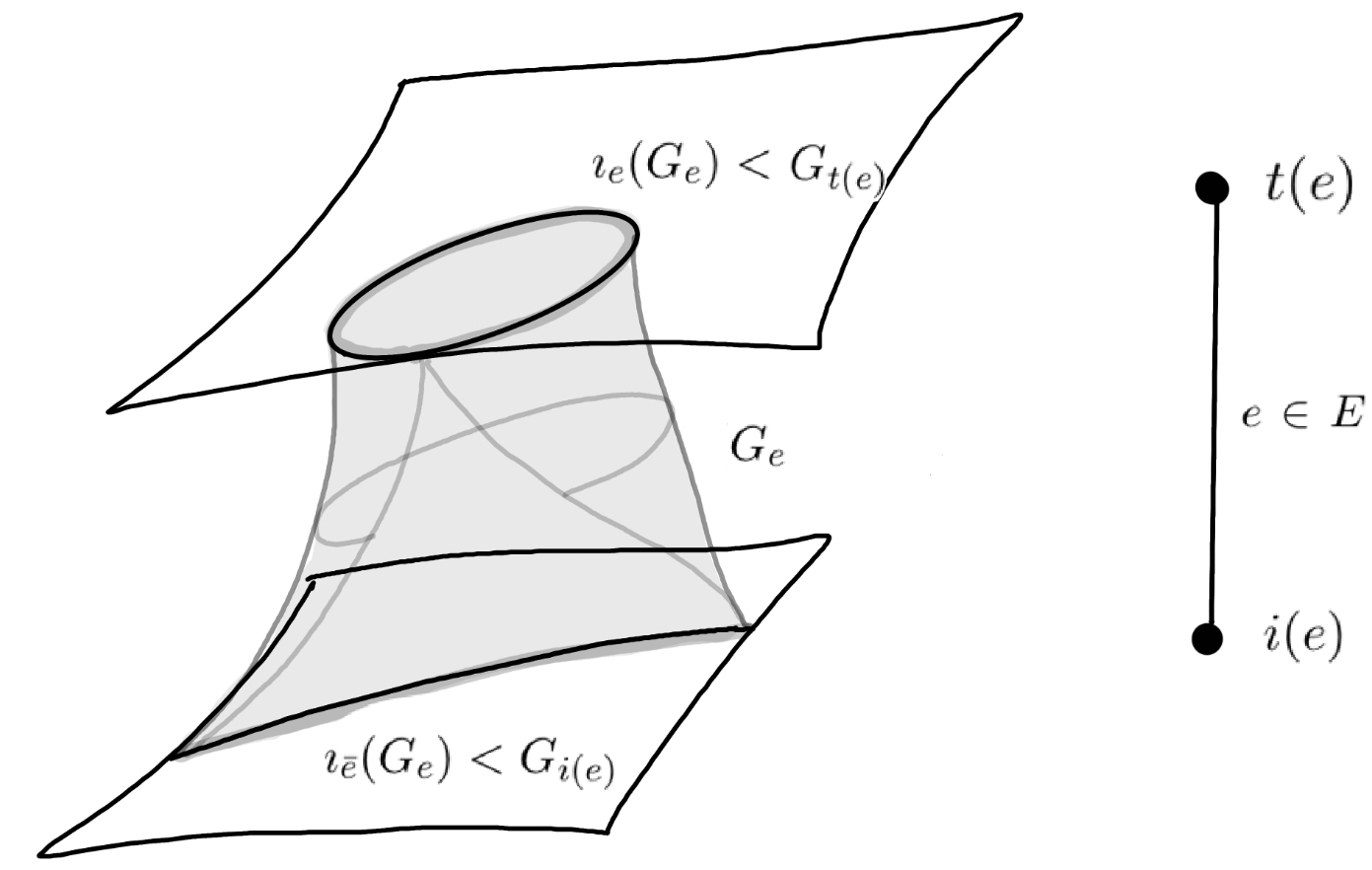}
    \caption{The domain and range of the crossing edge map, here distorting the edge subgroup along the way.}
    \label{Crossing edge}
    \end{figure}
The {\em $\mathcal G$-free product}, denoted by ${\mathcal G}_\Gamma$, is the free product of the vertex groups with the free group generated by the unonriented edges of $\Gamma$, namely
$${\mathcal G}_\Gamma=\ast_{v\in V}G_v\ast\FM_E.$$
A {\em (finite) $\mathcal G$-sequence} from $v \in V$ to $v^\prime \in V$ is a sequence 
$$s=(g_0,e_1,g_1,\cdots,e_k,g_k)$$
where $k \geq 0$, the number of edges in $s$, is the {\em edge-length of $s$}, and $e_1 \cdots e_k$ is an
edge-path in $\Gamma$ from $v$ to $v^\prime$. This is the {\it vertical} part of the sequence $s$, whereas $g_0 \in G_v$, $g_k \in G_{v^\prime}$ and $g_j \in G_{t(e_j)} = G_{i(e_{j+1})}$ for $j=1,\cdots,k-1$ form the {\it horizontal} part (despite a planar analogy, horizontal and vertical travels do in general not commute). A $\mathcal G$-sequence is called a {\it $\mathcal G$-loop} if it starts and ends at the same vertex of $V$. The concatenation of $\mathcal G$-sequences is defined in the obvious way. 

We write $s \sim t$ when two $\mathcal G$-sequences are equivalent
for the relation generated by the following elementary equivalences 
\begin{eqnarray*}(g,e,g^\prime) &\sim & (g \imath_{\e}(h), e, \imath_e(h^{-1}) g^\prime)\\
(g,e,1,\e, k) &\sim& (gk)\end{eqnarray*} 
where $e\in E$, $h\in G_e$ and $g, k\in G_{i(e)}, g'\in G_{t(e)}$. We then denote by $[s]$
the equivalence class of a $\mathcal G$-sequence $s$.
\end{definition}
\begin{definition}
Let $\mathcal G$ be a graph of groups and choose a base-vertex $v_0 \in V$, the {\it fundamental group} $\pi_1(\mathcal G)$ is the group of equivalence classes of $\mathcal G$-loops based at $v_0$.
\end{definition} 
One checks that this is a group for the law given by the concatenation: $[s] [t] := [st]$, the neutral element is the class of the length one sequence consisting of the neutral element, and the inverse of $[s]$ is given by the loop obtained reading any representative $s$ backwards and inverting the elements at each step. Note that if we change the base vertex then we get an isomorphic fundamental group, hence we omit the base-vertex. 
\begin{remark} Let $T\subseteq\Gamma$ be a maximal tree, then the fundamental group $\pi_1(\mathcal G)$ is the following quotient
$$\pi_1(\mathcal G)={\mathcal G}_\Gamma/R$$
where $R$ is the following set of relations:
$$R=\{e=1, \imath_e(h) = \imath_{\e}(h) \hbox{ if }e \in E\cap T,\hbox{ and }\imath_e(h) = e^{-1} \imath_{\e}(h) e\hbox{ if }e \in E\setminus T\}.$$
Indeed, a {\em reduction} of a finite $\mathcal G$-sequence $s$ consists of the substitution of a subsequence in $s$ of the form $(g,e,\imath_e(h),\e,g^\prime)$
by the element $g \imath_{\e}(h) g^\prime\in G_{t(e)}$. If $t$ is the resulting $\mathcal G$-sequence, then $s \sim t$. A $\mathcal G$-sequence is
{\em reduced} if no reduction is possible. Any $\mathcal G$-sequence is equivalent to a reduced one. More precisely, if $s =(g_0,e_1,g_1,\cdots,e_k,g_k)$ and
$t = (g^\prime_0,e^\prime_1,g^\prime_1,\cdots,e^\prime_{k^\prime},g^\prime_{k^\prime})$ are two equivalent reduced $\mathcal G$-sequences then $k=k^\prime$,
$e^\prime_j = e_j$ for $j=1,\cdots,k$ and, if we denote by $s_j$ (resp. $t_j$) the initial subpath of $s$ (resp. $t$) ending with $e_j$ for $j \geq 1$ ($s_0,t_0$ being trivial), then $s^{-1}_j t_j$ defines an element in $G_{t(e_j)}$ which belongs to $\imath_{\e_{j+1}}(G_{e_{j+1}})$
for $j=0,\cdots,k-1$.
\end{remark}
\begin{definition}\label{stanLength}Given a graph of groups $\mathcal G$, with fixed generating sets on the vertex groups, we will call {\it standard} set of generators for $\pi_1(\mathcal G)$ the one given by the union of the generators of the vertex-groups, along with the edges in $\Gamma$, which is also a generating set for ${\mathcal G}_\Gamma$. We will denote by $L_\Gamma$ (resp. $L_{\mathcal G}$) the length function on the group $\mathcal G_\Gamma$ (resp. $\pi_1(\mathcal G)$) obtained from that set of generators. Namely, for $s=(g_0,e_1,g_1,\cdots,e_k,g_k)$ a reduced ${\mathcal G}$-sequence, then $L_{\Gamma}(s)=k+\sum_{i=1}^kL_{G_{t(e_i)}}(g_i)$, so that $L_{G_v}(g)=L_{\Gamma}(g)$ for any $g\in G_v$ and hence:
$$L_{\mathcal G}(s) = \min\{n=k+\sum_{i=1}^kL_{G_{t(e_i)}}(g'_i)\ |\,s = (g'_0,e_1,g'_1,\cdots,e_k,g'_k)\}$$
Both, as well as the length functions associated with the vertex-groups (which agree with the restriction of the $\Gamma$-length function) will be termed {\em standard length functions}. Notice that if $\Gamma$ is a tree to begin with, then $\pi_1(\mathcal G)$ is the free product of the vertex-groups, amalgamated over the edge-groups. 
\end{definition}
\begin{definition}
\label{defpd}
A graph of groups $\mathcal G$ {\em has at most polynomial distortion} if for some (and hence any) set of standard length functions, there exists a polynomial $P$ such that for any vertex-group $G_v$ in ${\mathcal G}$, for any $g \in G_v$, for any $n \in \mn$, for any ${\mathcal G}$-loop $s$ of edge-length $n$ such that $[s]=g$, then the following inequality holds:
$$L_{G_v}(g) \leq P(n) L_{\Gamma}(s)$$
We shall say that $\mathcal G$ {\em is undistorted} if the polynomial $P$ is a constant.
\end{definition}
In words, polynomial distortion implies that traveling outside the vertex-group provides a polynomial shortcut at most, but it gives the more precise information that traveling horizontally provides no significant shortcut, and that traveling vertically provides a polynomial shortcut at most. We will prove that Definition \ref{defpd} is satisfied by the graph of groups emanating from $3$-manifolds groups which was the original motivation. We can however weaken a little bit the definition and still get the Rapid Decay property in the context of more general graph of groups:
\begin{definition}
\label{loosepd}
A graph of groups $\mathcal G$ {\em has loose polynomial distortion} if for some (and hence any) set of standard length functions, there exists a polynomial $P$ such that for any vertex-group $G_v$ in ${\mathcal G}$, for any $g \in G_v$, for any ${\mathcal G}$-loop $s$ such that $[s]=g$, then the following inequality holds:
$$L_{G_v}(g) \leq P(L_{\Gamma}(s)).$$
\end{definition}
\begin{remark}
\label{croissant}
Since $\displaystyle \sum^n_{i=0} a_i x^i \leq \sum^n_{i=0} |a_i| x^i$ for $x$ in $\mr^+$, we can substitute the polynomial $P$ of Definitions \ref{defpd} and \ref{loosepd} by a polynomial with same degree such that:
\begin{itemize}
 \item $P$ is increasing over $\mr^+$.
 \item For any non negative numbers $r_i$, $i=1,\cdots,l$, one has
 $$\sum^l_{i=1}P(r_i) \leq P(\sum^l_{i=1} r_i).$$
\end{itemize}
{\em These assumptions will be assumed for all the polynomial dealt with throughout the paper.}
\end{remark}
\begin{remark}
\label{defpratique} A ${\mathcal G}$-loop $s$ that satisfies $[s]=g\in G_v$ can be reduced to a single vertex and hence its corresponding edge-path in the graph $\Gamma$ has even length $n=2k$ and is contractible. For our purpose, an equivalent but more tractable characterization of a graph of groups $\mathcal G$ having at most polynomial distortion is to restrict to loops $s=pq^{-1}$ where $p$ and $q$ are any couple of reduced ${\mathcal G}$-sequences of edge-length $k$ starting at $v$, with same endpoint and such that $[pq^{-1}]=g$.

Let $D:{\mn}\to{\mn}$ be a non-decreasing function, and let $H<G$ be two finitely generated groups. If we take a generating set for $H$ that is contained in a generating set for $G$, then for any $h\in H$, then $L_H(h)\leq L_G(h)$, because there are some shortcuts in $G$ that are not available in $H$. We say that $H$ has {\it distortion at most $D$} if, for any $h\in H$, those shortcuts are controlled by $D$, namely
$$L_H(h)\leq D(L_G(h))L_G(h).$$
Gromov in \cite{GroAs} Chapter 3 defines the distortion function as 
$${\rm Disto(n)}:=\frac{diam_H(H\cap B_G(n))}{n}.$$
One checks that $2D(n)\geq {\rm Disto}(n/2)$, so that both definitions are equivalent.

For $\mathcal G$ a graph of groups with loose polynomial distortion, one notices that the vertex-groups are all at most polynomially distorted in $G=\pi_1({\mathcal G})$. Indeed, since
$$L_G(h)=\min\{L_{\Gamma}(w)|\ [w]=g\}$$
we can deduce polynomial distortion from Definition \ref{defpd} or \ref{loosepd}. The edge-groups can a priori have any distortion in their terminal vertex-group, as the Sol example below shows. \end{remark}
\begin{remark}
\label{apdlpd}
Since the integer $n$ in Definition \ref{defpd} is the edge-length (the length of the vertical part) of the loop, it implies the loose polynomial distortion of Definition \ref{loosepd}. The converse is not true: observe in particular that, in a graph of groups with at most polynomial distortion, each crossing edge map is a quasi isometric embedding (take $n=1$ in Definition \ref{defpd}) which does not hold in general in the setting of loose polynomial distortion (see ${\mathcal G}_5$ in the example below).
\end{remark}
\begin{example}
\label{1stexample}
Take the graph of group ${\mathcal G}_0$ with one single vertex $v$ and one single edge $e$, with $G_v=\FM_2=\left<x_1,x_2\right>$, $G_e=\FM_2=\left<a_1,a_2\right>$ and $\imath_e(a_1)=\imath_{\e}(a_1)=x_1$ whereas $\imath_e(a_2)=x_2$ and $\imath_{\e}(a_2)=x_2x^2_1$. Then ${\mathcal G}_0$ has at most polynomial distortion.

Notice that if, instead of the last embedding above we take $\imath_{\e}(a_2)=x_2^2$, then this creates a copy of the Baumslag-Solitar group BS$(1,2)$, which is the fundamental group of the graph of groups ${\mathcal G}_1$ where $G_v=\mz=\left<x\right>$ and $G_e=\left<a\right>$, with $\imath_e(a)=x$ and $\imath_{\e}(a)=x^2$ so that $\pi_1({\mathcal G}_1)=\left<x,e|exe^{-1}=x^2\right>$ and one can see that the distortion is exponential.

Consider now the Formanek-Procesi group (see \cite{Formanek}). It is the fundamental group of the graph of groups ${\mathcal G}_2$ with one vertex $v$ and two edges $e_1,e_2$ defined as follows: $G_v = \F{3}=\left<x_1,x_2,y\right>$, $G_{e_i} = \F{3} = \left< a_1,a_2,b \right>$ with embeddings $\imath_{e_j}(a_i)=x_i$ $\imath_{\e_j}(a_i)=x_i$, $\imath_{\e_j}(b)=y$ and $\imath_{e_j}(b)=yx_j$ for $j=1,2$. As ${\mathcal G}_0$, the graph ${\mathcal G}_2$ has at most polynomial distortion.

Let us now consider the amalgamated product of the fundamental groups of two Sol manifolds $\mz^2 \rtimes_A \mz$ where $A$ is any hyperbolics matrix of $SL_2(\mz)$. The associated graph of groups ${\mathcal G}_3$ has two vertices $v,w$, with vertex-groups $G_v = G_w = \mz^2 \rtimes_A \mz$, edge-group $G_e = \mz^2$ and is at most polynomially distorted although the edge-group is exponentially distorted in both vertex-groups. One builds a similar example ${\mathcal G}_4$ by amalgamating two copies of the Baumslag-Solitar group $BS(1,2)$ along the subgroup $\mz$ which is exponentially distorted in both.

Let us now consider a polynomially, non linearly, growing automorphism $\alpha$ of $\F{3}=\left<x_1,x_2,x_3\right>$, for instance $\alpha(x_1) = x_1 x^2_2 x^3_3, \alpha(x_2) = x_2 x^4_3$, $\alpha(x_3)=x_3$. Let $G = \F{3} \rtimes_\alpha \mz$ and $H = \F{3} * \mz$. The graph of groups ${\mathcal G}_5$ with two vertices $v,w$ with vertex-groups $G_v=G$ and $G_w=H$, one single edge $e$ with edge-group $G_e=\F{3}$ with the obvious embeddings of $G_e$ into $G_v$ and $G_w$ has loose polynomial distortion but does not have at most polynomial distortion.
\end{example}
We are now defining the {\em universal covering of a graph of group}, also known as {\em Bass-Serre tree}, which is a tree with vertices and edges labelled by cosets of vertex groups. 
\begin{definition}
Let $\mathcal G =(\Gamma,\{G_e\},\{G_v\},\{\imath_e\})$ be a graph of groups. The {\em universal covering of $\mathcal G$} is the labelled tree ${\mathcal T}_{\mathcal G} = (\widetilde{\Gamma},\varphi)$ where $\widetilde{\Gamma}=(\widetilde{V}, \widetilde{E})$ is a tree defined as follows:
\begin{enumerate}
\item The vertices $\tilde{v}\in\widetilde{V}$ are equivalence classes of $\mathcal G$-sequences $s$ in $\mathcal G$ starting at $v_0$ and ending at $v$, namely $$\tilde{v}=sG_v=\{sg \hbox{ a $\mathcal G$-sequence with } g \in G_v\}.$$
\item Two vertices $\tilde{v},\tilde{v}^\prime\in\widetilde{V}$ are connected by an edge $\tilde{e}=(\tilde{v},\tilde{v}^\prime)\in\widetilde{E}$ with $i(\tilde{e})=\tilde{v}$, $t(\tilde{e})=\tilde{v}^\prime$ if and only if there is a $\mathcal G$-sequence $s$ from $v_0$ to $v=i(e)$ and there is $h \in G_v$ such that if $\tilde{v}= s G_v$, then $\tilde{v}^\prime = (s, h, e) G_{v^\prime}$ with $v^\prime = t(e)$.
\end{enumerate}
The base-point of ${\mathcal T}_{\mathcal G}$ is $x_0 = 1_GG_{v_0}$.
The labelling $\varphi:\widetilde{\Gamma}\to {\rm Cosets}(\pi_1({\mathcal G}))$ assigns to vertices in $\widetilde{V}$ the label with the terminal vertex in $\mathcal G$ of the corresponding $\mathcal G$-sequence: $\varphi(\tilde{v})= sG_v$ if $v$ is this terminal vertex and $s$ the $\mathcal G$-sequence starting at $v_0$ and ending at $v$. For the edge $\tilde{e}$ between $\tilde{v}=sG_v$ and $\tilde{v}'=(s,h,e)G_{v'}$ with $e\in E$, the label is given by $\varphi(\tilde{e})=(s, h) G_e$. There are then maps
$$\varphi({\imath}_{\tilde{e}})=\widetilde{\imath}_e=(s,h) \imath_{e} \colon {G}_{{e}} \rightarrow \varphi({t(\tilde{e})})=(s,h,e)G_{v'}$$
embedding the edge-groups of ${\mathcal G}$ in the cosets of the vertex-groups in $\pi_1({\mathcal G})$, namely the vertices of the universal cover ${\mathcal T}_{\mathcal G}$.
\end{definition}
\begin{remark}
 Observe that ${\mathcal T}_{\mathcal G}$ admits a natural left-action of $\pi_1(\mathcal G)$ by pre-composition by a loop based at $v_0$ of the ${\mathcal G}$-sequence. The stabilizers of the vertices (resp. of the edges) under this action are conjugates of the vertex-groups (resp. edge-groups) of $\mathcal G$. We denote by $\pi_{\mathcal G} \colon {\mathcal T}_{\mathcal G} \rightarrow \mathcal G$ the covering-map, i.e. the projection which identifies the vertices and edges of $\widetilde{\Gamma}$ with same label in $\Gamma$. The terminology for graph of groups lifts naturally to the universal covering. If $s$ is a ${\mathcal T}_{\mathcal G}$-sequence then $\pi_{\mathcal G}(s)$ is a $\mathcal G$-sequence. The elements of the fundamental group are identified to the ${\mathcal T}_{\mathcal G}$-sequences from $x_0$ to another lift in $\pi^{-1}_{\mathcal G}(v_0)$, since they project under $\pi_{\mathcal G}$ to $\mathcal G$-loops based at $v_0$. The crossing edge map $c_e$ associated to an edge $e$ from Definition \ref{gog} lifts to a crossing edge map $\tilde{c}_{\tilde{e}}=\widetilde{\imath}_e\circ{\widetilde\imath}_{\e}^{-1}$ allowing to move the subset $s\widetilde{\imath}_{\e}(G_e)$ of the coset $sG_{i(e)}$ into the coset $(s,e)G_{t(e)}$, and satisfying $\tilde{c}_e\circ\tilde{c}_{\e}=id$.
\end{remark}
When given a set of generators $S$ of a group $G$, a {\em geodesic word} for this set of generators is a word $w$ in $S \cup S^{-1}$ whose number of letters equals the length of the corresponding element for the length function on $G$ associated to $S$. %
\begin{definition}
Let $\mathcal G = (\Gamma,\{G_e\},\{G_v\},\{\imath_e\})$ be a graph of groups together with a fixed choice of generating sets of the edge-groups. We endow $\pi_1({\mathcal G})$ and $\mathcal G_{\Gamma}$ with the corresponding standard length functions (Definition \ref{stanLength}). Let ${\mathcal T}_{\mathcal G}=(\widetilde{\Gamma},\varphi)$ be the universal covering of $\mathcal G$. 
\begin{enumerate}
  \item A {\em $\mathcal G$-path} is a $\mathcal G$-sequence $s=(w_0,e_1,w_1,\cdots,e_k,w_k)$ where $k \geq 0$ is an integer, each $w_i$ is a geodesic word for the standard length function of the corresponding vertex-group in ${\mathcal G}$, each $e_i$ is an edge in ${\mathcal T}_{\mathcal G}$ such that the sequence $e_1\cdots e_k$ is a path in the tree $\widetilde{\Gamma}$. The $\mathcal G$-path is called {\it reduced} if the corresponding path in the tree $\widetilde{\Gamma}$ is reduced.
  \item If $p$ is a $\mathcal G$-path, the {\em $\Gamma$-length of $p$}, denoted by $|p|_{\Gamma}$ is equal to the sum of the $G_v$-lengths of the vertex-elements in $p$ with the number of edges in $p$. A {\em $\mathcal G$-geodesic} is any $\mathcal G$-loop $p$ such that $|p|_{\Gamma} = L_{\mathcal G}([p])$.
  \item A {\em normal $\mathcal G$-set} is a set of $\mathcal G$-geodesics in bijection with $\pi_1(\mathcal G)$.
  \item A {\em normal $\Gamma$-set} is a set of reduced $\mathcal G$-paths, in bijection with a normal $\mathcal G$-set, and so that any of its elements is constructed from the $\mathcal G$-geodesics in this normal $\mathcal G$-set by iteration of the following process: apply a reduction at a geodesic word and substitute the resulting word by a geodesic word.
  \item If $\mathcal G$ is equipped with normal $\mathcal G$- and $\Gamma$-sets, for each element $g \in \pi_1(\mathcal G)$, we denote by $\widetilde{g}$ (resp. $\widehat{g}$) the element of the normal $\mathcal G$-set (resp. $\Gamma$-set) which represents $g$.
  \end{enumerate}
 \end{definition}
While $\mathcal G$-sequences belong to the free product ${\mathcal G}_{\Gamma}$, the subset of $\mathcal G$-paths lie in the universal covering ${\mathcal T}_{\mathcal G}$ of $\mathcal G$.
Moreover, for $\mathcal G$ equipped with normal $\mathcal G$- and $\Gamma$-sets, then each element of the normal $\Gamma$-set is a geodesic word for the $L_\Gamma$-length function. Indeed, the $\Gamma$-length function in restriction to any vertex-group agrees with the standard length function of this vertex-group and the elements of the normal $\Gamma$-set are reduced. Finally, notice that a reduced $\mathcal G$-path is a $\mathcal G_{\Gamma}$-geodesic, but is in general not a $\mathcal G$-geodesic.
\begin{lemma}
\label{consequence 1}
Let $\mathcal G = (\Gamma,\{G_e\},\{G_v\},\{\imath_e\})$ be a graph of groups with loose polynomial distortion. Then there is a polynomial $Q$ such that, for any $g \in \pi_1(\mathcal G)$ the following inequality is satisfied:
$$L_\Gamma(\widehat{g}) \leq Q(L_{\mathcal G}(g)),$$
where $\widehat{g}\in{\mathcal G}_{\Gamma}$ is the element of the normal $\Gamma$-set representing $g$.
\end{lemma}
\begin{proof}
By definition of a normal $\mathcal G$-set, for each $g \in \pi_1(\mathcal G)$ one has $L_{\mathcal G}(g) = |\widetilde{g}|_\Gamma$. The normal $\mathcal G$-path $\widetilde{g}$ admits a (non-unique) collection of subpaths $r_1,\cdots,r_l$ with the following properties:
\begin{itemize}
  \item each $r_i$ can be reduced to an element in a vertex-group, 
  \item if $i \neq j$, either $r_i \cap r_j$ is empty or consists of a single vertex,
  \item after applying all the necessary reductions to each one of the $r_i$'s, one gets a reduced $\mathcal G$-path.
\end{itemize}
Let us denote by $R^c=\widetilde{g}\setminus\cup r_i$, namely the union of all the subpaths in the complement in $\widetilde{g}$ of the union of the $r_i$'s in such a collection, which allows the following decomposition:
$$|\widetilde{g}|_\Gamma = \sum^l_{i=1} |r_i|_\Gamma + |R^c|_\Gamma.$$
Since $\mathcal G$ has loose polynomial distortion, after applying all the reductions at the $r_i$'s, one gets a reduced $\mathcal G$-path $g_0$ such that:
$$|g_0|_\Gamma \leq \sum^l_{i=1} P(|r_i|_\Gamma) + |R^c|_\Gamma.$$
where $P$ is the polynomial in Definition \ref{loosepd}. Since the $|r_i|_\Gamma$ are positive integers and by Remark \ref{croissant}, $$\sum^l_{i=1} P(|r_i|_\Gamma) \leq P(\sum^l_{i=1} |r_i|_\Gamma).$$  Of course $\sum^l_{i=1} |r_i|_\Gamma \leq |\widetilde{g}|_\Gamma$ so that we get:
$$|g_0|_\Gamma \leq  P(|\widetilde{g}|_\Gamma) =P(L_{\mathcal G}(g)).$$
The proof above holds for any choice of a collection of subpaths $r_i$'s. By definition, the element $\widehat{g}$ in the chosen normal $\Gamma$-set corresponds to a particular choice of such a collection hence the conclusion since $L_{\Gamma}(\widehat{g})=|g_0|_{\Gamma}$.
\end{proof}
\subsection{The Rapid Decay property}\label{RD}
We refer to \cite{ChatterjiIntro} or \cite{indirathese} for an introduction to the Rapid Decay property. We just give here a rough account of the material that we need. For $A = \mc, \mr^+$, the notation $AG$ denotes the set of finitely supported $A$-valued functions on the group $G$. The support of a function $f$ is denoted by supp$(f)$. The {\em convolution} of two functions $f,g \in AG$ is defined by 
$$(f * g)(\gamma) = \sum_{\mu \in G} f(\mu) g(\mu^{-1}\gamma),$$
which corresponds to the standard ring operation on $AG$. A group homomorphism $\varphi:G\to H$ induces a map 
$$A\varphi:AH\to AG, f\mapsto f\circ\varphi$$
by precomposition, that in general doesn't preserve any norm on $AG$ or $AH$, like the $\ell^2$-norm given by $||f||_2 =\sqrt{\sum_{\gamma \in G} |f(\gamma)|^2}$ for $f\in AG$.
\begin{definition}\label{defRD}
A group $G$ {\em has the Rapid Decay property with respect to a length function
$L$} if there exists a polynomial $P$ such that, for
any $r > 0$, for any $f \in \mr^+G$ with supp$(f) \subseteq B_L(r)$, and for any $g \in \mr^+G$
$$||f*g||_2 \leq P(r) ||f||_2 ||g||_2.$$
\end{definition}
\begin{remark}\label{obstruction}
Taking the supremum on both sides of the above inequality over all $g\in\mr^+$ on the ball of radius 1 for the $\ell^2$-norm gives the following equivalent condition:
$$||f||_* \leq P(r) ||f||_2$$
where now $\|f\|_*$ is the {\em operator norm} of $f$ acting on $\ell^2(G)$ via the left multiplication in $AG$. Leptin's characterization of amenability says that if $f$ is supported on an amenable subgroup of $G$, then $\|f\|_*=\|f\|_1$ and so, if $f$ is the characteristic function of a ball of radius $r$ on an amenable subgroup, and in particular a cyclic one, that prevents exponential distortion.
\end{remark}
\begin{remark}\label{tauto}Let $\mathcal G = (\Gamma,\{G_e\},\{G_v\},\{\imath_e\})$ be a graph of groups together with a normal $\Gamma$-set. For $f \in \mr^+ \pi_1(\mathcal G)$, let us define $\widehat{f} \in \mr^+ {\mathcal G}_\Gamma$ by $\widehat{f}(\widehat{g}) = f(g)$ and $\widehat{f}(w) = 0$ for any element $w \in {\mathcal G}_\Gamma$ not defined by an element in the normal $\Gamma$-set.  If $f \in\mr^+\pi_1(\mathcal G)$ then $\widehat{f} \in \mr^+{\mathcal G}_\Gamma$ with $||\widehat{f}||_{2} = ||f||_2$. Moreover, if $\mathcal G$ has loose polynomial distortion (and so also if $\mathcal G$ has at most polynomial distortion by Remark \ref{apdlpd}) according to Lemma \ref{consequence 1} there is some polynomial $Q$ such that for any $f \in \mr^+ \pi_1(\mathcal G)$ with supp$(f) \subseteq B_{L_{\mathcal G}}(r)$, then supp$(\widehat{f}) \subseteq B_{L_\Gamma}(Q(r))$.
\end{remark}
The following lemma provides an isometric map $\mr^+ \pi_1(\mathcal G)\to\mr^+ {\mathcal G}_\Gamma$ preserving the $\ell^2$-norms of products.  %
\begin{lemma}
\label{magic}
Let $\mathcal G = (\Gamma,\{G_e\},\{G_v\},\{\imath_e\})$ be a graph of groups with at most polynomial distortion. For any functions $f, g \in \mr^+ \pi_1(\mathcal G)$ with supp$(f) \subseteq B_{L_{\mathcal G}}(r)$ and $g \in\ell^2(\pi_1(\mathcal G))$, there exist a polynomial $P$ and functions $F, G \in \mr^+ {\mathcal G}_\Gamma$ with supp$(F) \subseteq B_{L_{\Gamma}}(P(r))$ and $G \in\mr^+{\mathcal G}_\Gamma$ such that $||f * g||_2 = ||F * G||_2$, $||f||_2 = ||F||_2$ and $||g||_2 = ||G||_2$.
\end{lemma}
\begin{proof}
We equip $\mathcal G$ with normal $\mathcal G$- and $\Gamma$-sets. Let us fix $u \in \pi(\mathcal G)$. For each $v \in supp(f) \subseteq B_{L_{\mathcal G}}(r)$, we set $F(\widehat{v}) = \widehat{f}(\widehat{v})= f(v)$ and $G(\widehat{v}^{-1} \widehat{u}) = \widehat{g}(\widehat{v^{-1}u})=g(v^{-1} u)$.  Since $\widehat{v}^{-1} \widehat{u} = \widehat{v^\prime}^{-1} \widehat{u^\prime}$ in ${\mathcal G}_\Gamma$ implies $v^{-1} u = {v^\prime}^{-1} u^\prime$ in $\pi(\mathcal G)$, there is no obstruction in so defining $F$ and $G$ and extending them by zero on the elements of ${\mathcal G}_\Gamma$ where they are not been defined. Hence, for each  $u \in \pi_1(\mathcal G)$ we have two functions $F, G$ on the $\mathcal G$-free product such that $\widehat{f * g}(\widehat{u}) = (F*G)(\widehat{u})$ and the lemma follows from Remark \ref{tauto}.
\end{proof}
We are now ready to prove Proposition \ref{distortion} announced in the Introduction:
\begin{proof}[Proof of Proposition \ref{distortion}]Let $\mathcal G = (\Gamma,\{G_e\},\{G_v\},\{\imath_e\})$ be a graph of groups which has loose polynomial distortion.
Let $f \in \mr^+ \pi_1(\mathcal G)$ with supp$(f) \subseteq B_{L_{\mathcal G}}(r)$ for some $r > 0$ and let $g \in \mr^+ \pi(\mathcal G)$. Lemma \ref{magic} gives a polynomial $P$ and functions $F, G \in \mr^+ {\mathcal G}_\Gamma$ with supp$(F) \subseteq B_{L_{\Gamma}}(P(r))$ and $G \in\ell^2({\mathcal G}_\Gamma)$ such that $||f * g||_2 = ||F * G||_2$, $||f||_2 = ||F||_2$ and $||g||_2 = ||G||_2$. By \cite{Jolissaint}, the $\mathcal G$-free product ${\mathcal G}_\Gamma$ has Rapid Decay property. Hence there is a polynomial $Q$ such that $||F * G||_2 \leq Q(P(r)) ||F||_2 ||G||_2$. Since $||f * g||_2 = ||F * G||_2$, $||f||_2 = ||F||_2$ and $||g||_2 = ||G||_2$, we get the announced conclusion.
\end{proof}
\begin{corollary}
Let $\mathcal G = (\Gamma,\{G_e\},\{G_v\},\{\imath_e\})$ be a graph of groups such that $\Gamma$ is a finite tree, the edge-groups are cyclic and each vertex-group has Rapid Decay property. Then $\mathcal G$ has loose polynomial distortion, and hence $\pi_1({\mathcal G})$ has Rapid Decay property as well.
\end{corollary}
\begin{proof}
To begin with, assume that $\Gamma$ consists of a single edge $e$, with $v=i(e)$ and $w=t(e)$ distinct. Without loss of generality (up to adding generators in the generating sets of the vertex groups), we may assume that the images under $\imath_e$ and $\imath_{\e}$ of the generator $z$ of the cyclic edge-group belong to the generating sets of $G_v$ and $G_w$. Let $x_1=\imath_e(z)$ be the image of that generator in $G_v$ and $y_1=\imath_{\e}(z)$ be the one in $G_w$. Then, since the morphisms $\imath_e$ are injective, there are integers $p, q$ such that $c_e(x^p_1)=y^q_1$, where $c_e$ is the crossing map for the edge $e$. The cyclic subgroups generated by $x_1$ and $y_1$ are at most polynomially distorted in their respective vertex groups as those have Rapid Decay property\footnote{Since $\mz$ is amenable, an exponentially distorted copy of $\mz$ is an obstruction to the Rapid Decay property, see Remark \ref{obstruction}.}. Hence there is an exponential contraction or dilatation if and only if $|p| \neq |q|$. Since the edge-groups are cyclic, such a distortion only occur at powers of $x_1$ or $y_1$, depending on the lift of vertex $v$ or $w$ considered. Since $v \neq w$, in any ${\mathcal T}_{\mathcal G}$-sequence, the lifts of $v$ and $w$ alternate. Hence a dilatation is followed by a contraction. It follows that $\mathcal G$ has loose polynomial distortion.

For the general case, there is more than a single edge in the tree $\Gamma$, but when passing to a lift of $\Gamma$ to another one in a ${\mathcal T}_{\mathcal G}$-sequence, the $\mathcal G$-sequence has the form $(\cdots, e,\e, \cdots)$ for some oriented edge $e$ of $\Gamma$. Hence a contraction is followed by a dilatation, and conversely, so that the graph of groups ${\mathcal G}$ has loose polynomial distortion.

With $\mathcal G$ having loose polynomial distortion, Proposition \ref{distortion} gives the conclusion since the vertex-groups have Rapid Decay property. 
\end{proof}
\section{Ensuring polynomial distortion}\label{PolDisto}
We introduce below a couple of conditions which together imply loose or at most polynomial distortion.
\begin{definition}
\label{defspd}
A graph of groups $\mathcal G$ {\em seemingly has at most polynomial distortion} if for some (and hence any) set of standard length functions there exists a polynomial $P$ such that for any vertex-group $G_v$ in ${\mathcal G}$, for any $g \in G_v$, for any $n \in \mn$, for any reduced edge-path $p$ in $\Gamma$ with edge-length $n$ starting at $v$, for any $h \in G_{t(p)}$ such that $[php^{-1}]=g$, the following inequality holds:
$$L_\Gamma(g) \leq P(n) L_{\Gamma}(p h p^{-1})$$
We say that $\mathcal G$ is {\em seemingly undistorted} if the polynomial $P$ is a constant.
\end{definition}
And of course we have the following weaker notion.
\begin{definition}
\label{deflspd}
A graph of groups $\mathcal G$ {\em seemingly has loose polynomial distortion} if for some (and hence any) set of standard length functions there exists a polynomial $P$ such that for any vertex-group $G_v$ in ${\mathcal G}$, for any $g \in G_v$, for any reduced edge-path $p$ in $\Gamma$ starting at $v$, for any $h \in G_{t(p)}$ such that $[php^{-1}]=g$, the following inequality holds:
$$L_\Gamma(g) \leq P(L_{\Gamma}(p h p^{-1})).$$
\end{definition}
To seemingly have polynomial distortion is strictly weaker than actually having polynomial distortion, because being allowed $\mathcal G$-sequences is much more general than edge-paths only, as the following example shows. 
\begin{example}\label{SeemPolbutExpDisto}
Take $\Gamma$ the tree with two vertices $v, w$ and a single edge $e$, with $i(e)=v$ and $t(e)=w$, with $G_v=\left<x_1,x_2,x_3\,|\ x_2=x_3x_1x_3^{-1}\right>$ and $G_w=\left<y_1,y_2,y_3\,|\ y_1=y_3y_2y_3^{-1}\right>$ whereas $G_e= \FM_2=\left<a,b\right>$ with the following maps
\begin{eqnarray*}\imath_e(a)=y_1,&\ &\imath_e(b)=y_2^2,\\
\imath_{\e}(a)=x_1^2,&\ &\imath_{\e}(b)=x_2.\end{eqnarray*}
Hence ${\mathcal G}_\Gamma=G_v *G_w*\mz$ and let us see that the distortion is exponential, so more than polynomial. One has $x^{2^n}_1 = a y^{2^{n-1}}_1 a^{-1}$ and  $y^{2^{n-1}}_1 =  y_3 y^{2^{n-1}}_2 y^{-1}_3$ and then using the edge $y^{2^{n-1}}_2 =  a^{-1} x^{2^{n-2}}_2 a$ and $x^{2^{n-2}}_2 = x_3 x^{2^{n-2}}_1 x^{-1}_3$. Combining these last two relations gives $x^{2^n}_1 = a y_3 a^{-1} x_3 x^{2^{n-2}}_1 (a y_3 a^{-1} x_3)^{-1}$ hence an exponential distortion of $x^{2^n}_1$ by iterating the process since 
$$L_\Gamma(a y_3 a^{-1} x_3 x^{2^{n-2}}_1 (a y_3 a^{-1} x_3)^{-1}) =\frac{1}{4} L_\Gamma(x^{2^n}_1) + 8.$$
But it is seemingly polynomially distorted because the edge-paths that do the distortion, do it on at most one edge.\end{example} 
This example prompts the following definition.
\begin{definition}\label{well} Let $\mathcal G = (\Gamma,\{G_e\},\{G_v\},\{\imath_e\})$ be a graph of groups equipped with a set of standard length functions, and let ${\mathcal T}_{\mathcal G}= (\widetilde{\Gamma},\varphi)$ be its universal covering. 
If $p=e_1 \cdots e_l$ is an edge-path in $\widetilde{\Gamma}$ and $a \in \widetilde{G}_{i(p)}$, we say that the path $p$ is {\em well-defined at $a$} if the concatenation of the consecutive edge crossings 
$$\widetilde{c}_p(a):=(\widetilde{c}_{e_l} \circ \cdots \circ \widetilde{c}_{e_1})(a)$$
exists. If $e$ is an edge of $\widetilde{\Gamma}$ with $i(e)=t(p)$, we say that the path $p$ is {\it maximal at $a$ with respect to $e$} if $p$ is well-defined at $a$ but the path ${p\,e}$ is not.
\begin{figure}[htpb]
    \centering
        \includegraphics[width=\linewidth]{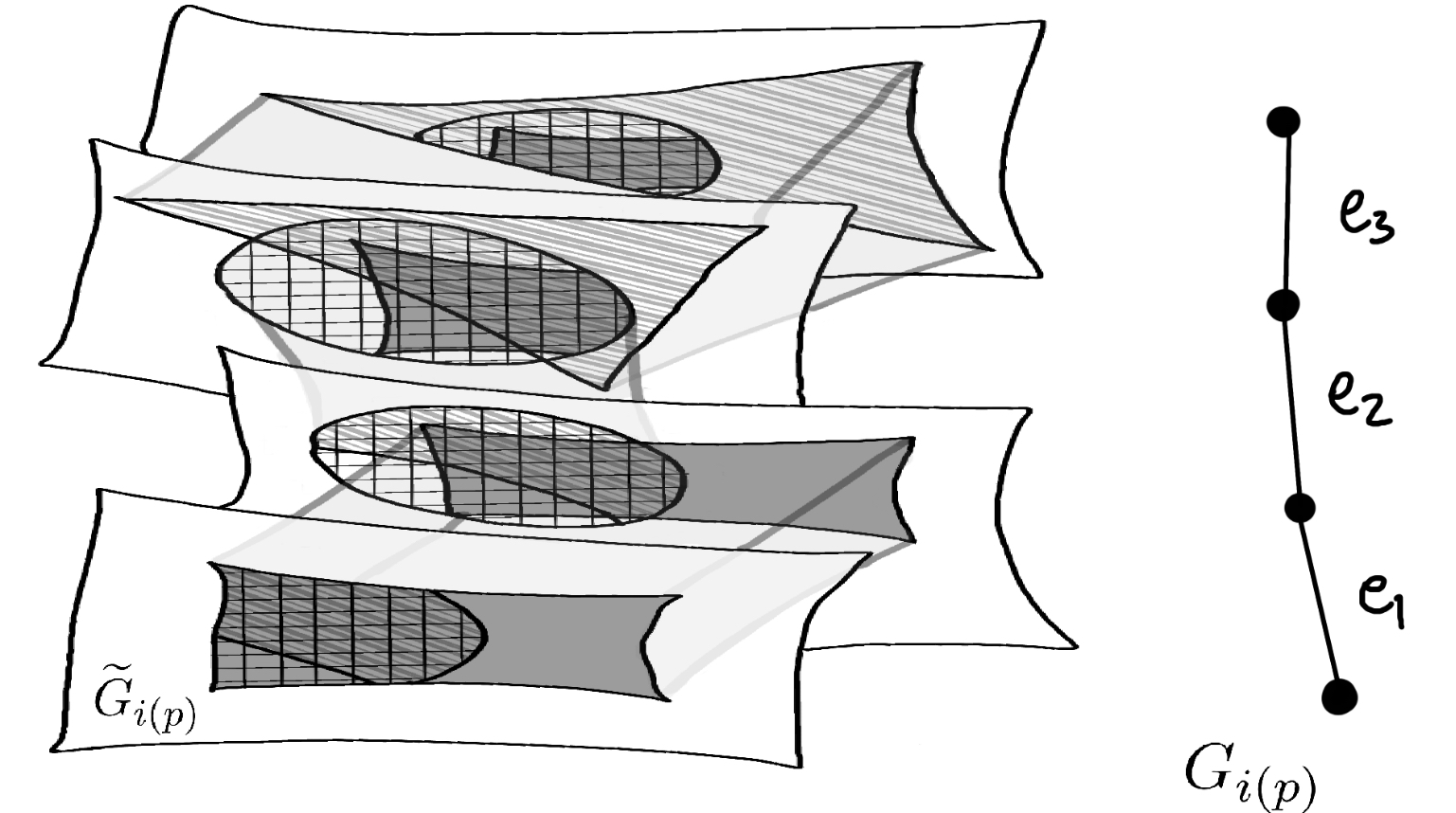}
    \caption{The path $p=e_1e_2e_3$ is well-defined at the triply shaded area of $\widetilde{G}_{i(p)}$. The path $p'=e_1e_2$ is well-defined at the grid and solid shaded area of $\widetilde{G}_{i(p)}$. The part without the line shading is maximal with respect to $e_3$.}
    \label{welldef}
    \end{figure}
\end{definition}
Example \ref{SeemPolbutExpDisto} is a particular case of the following lemma.
\begin{lemma}
\label{arbrespd}
Let $\mathcal G = (\Gamma,\{G_e\},\{G_v\},\{\imath_e\})$ be a graph of groups so that $\Gamma$ is a finite tree and so that, for any oriented edge $e$, the image of the corresponding crossing edge map distorts lengths at most polynomially, namely
$$L_{G_{i(e)}}(g)\leq P(L_{G_{t(e)}}(c_e(g)),$$
for any $g\in\imath_{\e}(G_e)<G_{i(e)}$. Then $\mathcal G$ seemingly has loose polynomial distortion. If moreover the crossing edge maps are quasi-isometric embeddings, then $\mathcal G$ is seemingly undistorted.
\end{lemma}
\begin{proof} Take $g\in G_v$ an element in some edge-group, and assume that $g=[php^{-1}]$ where $p$ is an edge-path in $\Gamma$ and $h\in G_{w}$ some other vertex-group. Since for any path $q$ the crossing maps satisfy $\tilde{c}_q\circ\tilde{c}_{q^{-1}}=Id$ and our path $p$ is purely vertical, we can assume it doesn't have any backtracking so has length at most $N$, the diameter of the tree $\Gamma$.

We now work in the universal covering of $\mathcal G$, and so consider two elements $u_0$ and $v_0$ in a lift of $i(p_1)$ such that $v^{-1}_0 u_0 = g$. The two elements $u = \widetilde{c}_{p}(u_0)$ and $v = \widetilde{c}_{p}(v_0)$ such that $v^{-1} u = h$ satisfy
$$L_\Gamma(v^{-1} u) \leq P^N(L_\Gamma(v^{-1}_0 u_0))$$ 
where $P$ is the maximal distortion polynomial for the crossing edge maps. That is $$L_\Gamma(h) \leq P^N(L_\Gamma(g)).$$ Since $g=[php^{-1}] \Leftrightarrow h=[p^{-1}gp]$ we also have $$L_\Gamma(g) \leq P^N(L_\Gamma(h)).$$ Since $P$ is increasing over $\mr^+$ (see Remark \ref{croissant}) and $L_\Gamma(h) \leq L_\Gamma(php^{-1})=2 L_\Gamma(p) + L_\Gamma(h)$ we eventually get $$L_\Gamma(g) \leq P^N(L_\Gamma(php^{-1}))$$ and seemingly loose polynomial distortion is proved. The last assertion of the lemma comes from the fact that a composition of quasi-isometric embeddings is a quasi-isometric embedding.
\end{proof}
 If arbitrarily long paths are only defined on finite sets and all the crossing edges maps are polynomial, it implies seemingly polynomial distortion. However, all the distortion could be made on a single edge, as shown in the following example.
\begin{example}\label{oneedge}
Consider the graph of groups with one single edge $e$ and two vertices $v$ and $w$, with $G_v=\mz^2=\left<a,b|[a,b]=1\right>$ and $G_w=BS(1,2)=\left<x,y|yxy^{-1}=x^2\right>$ and $G_e=\mz=\left<t\right>$. The embeddings are given by $\imath_{e}(t)=x$ and $\imath_{\e}(t)=a$. Then the crossing map distorts lengths exponentially.
\end{example}
The property below, about the crossing edge maps in the universal covering of the graph of groups, will be used to treat graph manifolds.
\begin{definition}
\label{ggtd}
Let $\mathcal G = (\Gamma,\{G_e\},\{G_v\},\{\imath_e\})$ be a graph of groups equipped with a set of standard length functions, and let ${\mathcal T}_{\mathcal G}= (\widetilde{\Gamma},\varphi)$ be its universal covering. We say that $\mathcal G$ {\em has tight dynamics} if there exist constants $K \geq 0$ and $C \geq 1$ such that for any oriented edges $\tilde{e},\tilde{f}$ of $\widetilde{\Gamma}$ with $t(\tilde{e})=i(\tilde{f})$, for any elements $a,b$ in $\tilde{\imath}_{\bar{e}}(G_e)$ such that:
\begin{itemize}
\item $L_\Gamma(b^{-1}a) \geq K$ 
and
\item $\tilde{e}$ is maximal at $a$ or $b$ or both with respect to $\tilde{f}$
\end{itemize}
then
$$L_\Gamma(a^{-1}_{\tilde{f}} \tilde{c}_{\tilde{e}}(a)) + L_\Gamma(b^{-1}_{\tilde{f}} \tilde{c}_{\tilde{e}}(b)) \geq \frac{1}{C} \left(L_\Gamma(b^{-1}a)-L_\Gamma(b^{-1}_{\tilde{f}} a_{\tilde{f}})\right).$$
for $a_{\tilde{f}}$ and $b_{\tilde{f}}$ any closest elements respectively to $\tilde{c}_{\tilde{e}}(a)$ and $\tilde{c}_{\tilde{e}}(b)$ in $\widetilde{\imath}_{{\overline{f}}}(G_f)$.
\begin{figure}[htpb]
    \centering
        \includegraphics[width=\linewidth]{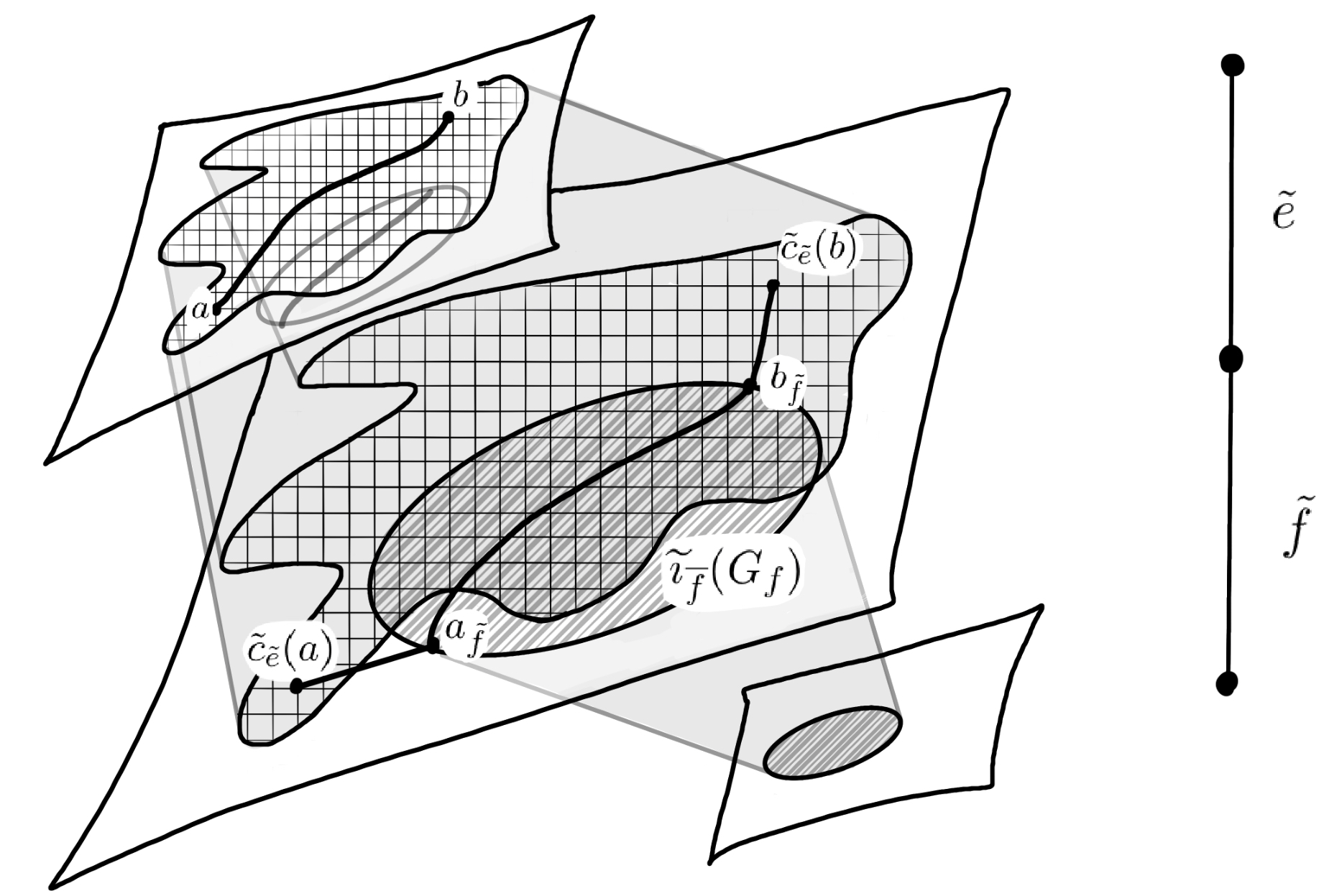}
    \caption{Tight dynamics: The distance between $a_{\tilde{f}}$ and $\tilde{c}_{\tilde{e}}(a)$ plus the one between $b_{\tilde{f}}$ and  $\tilde{c}_{\tilde{e}}(b)$ has to be larger than a fixed fraction of the difference of the distances between $a$ and $b$ with the one between $a_{\tilde{f}}$ and $b_{\tilde{f}}$.}
    \label{tight}
    \end{figure}
\end{definition}
Notice that if $L_\Gamma(b^{-1}a) \leq L_\Gamma(b^{-1}_{\tilde{f}} a_{\tilde{f}})$, then the above inequality is automatically satisfied.
The following lemma justifies the introduction of these notions.
\begin{lemma}
\label{spdtdpd}
A graph of groups with tight dynamics and seemingly (loose) polynomial distortion has (loose) polynomial distortion (with same degree). 
\end{lemma}
\begin{proof}
Let ${\mathcal T}_{\mathcal G}= (\widetilde{\Gamma},\varphi)$ be the universal covering of the graph of groups $\mathcal G = (\Gamma,\{G_e\},\{G_v\},\{\imath_e\})$. Let us consider $g \in G_v$ for some $v\in V$ and two reduced ${\mathcal T}_{\mathcal G}$-sequences $r,s$ of edge-length $n$ with $g = [r s^{-1}]$. By using the equivalence relations on ${\mathcal T}_{\mathcal G}$-sequences, we get edge-paths $p_i$ in $\widetilde{\Gamma}$, elements $a_i, b_i \in G_{i(p_i)}$, $1 \leq i \leq l$, and (possibly trivial) elements $a_0, b_0 \in G_{i(p_1)}$, $a_{l+1},b_{l+1} \in G_{t(p_l)}$ and an edge-path $p_{l+1}$ with the following properties (see Definition \ref{well}):
\begin{itemize}
  \item $i(p_{i+1})=t(p_i)$ and $t(p_{l+1}) = t(r)$,
  \item $p_{l+1}$ is well-defined at both $a_{l+1}$ and $b_{l+1}$,
  \item $p_i$ is maximal both at $a_i$ and at $b_i$ with respect to the first edge in $p_{i+1}$ for $1 \leq i \leq l$.
\end{itemize}
By setting
$${\bf p} = (a^{-1}_1 a_0, p_1, a^{-1}_2 \widetilde{c}_{p_1}(a_1), \cdots, a^{-1}_{l+1} \widetilde{c}_{p_l}(a_l), p_{l+1}, \widetilde{c}_{p_{l+1}}(b_{l+1})^{-1} \widetilde{c}_{p_{l+1}}(a_{l+1})$$
and
$${\bf q} = (b^{-1}_1 b_0, p_1, b^{-1}_2 \widetilde{c}_{p_1}(b_1), \cdots, b^{-1}_{l+1} \widetilde{c}_{p_l}(b_l), p_{l+1})$$
we get $$[{\bf p} {\bf q}^{-1}]=  [r s^{-1}] = g.$$
Since $\mathcal G$ seemingly has loose polynomial distortion, there is a polynomial $P$ such that
$$L_\Gamma({\bf p} {\bf q}^{-1}) \leq P(L_\Gamma(r s^{-1})).$$
We consider what happens at $G_{t(p_1)}$. Since $\mathcal G$ has tight dynamics, there are constants $K \geq 0$ and $C \geq 1$ such that, if $L_\Gamma(g) \geq K$ then $$L_\Gamma(g)-L_\Gamma(b^{-1}_2 a_2)\leq C \left(L_\Gamma(a^{-1}_1 a_0)+L_\Gamma(b^{-1}_1 b_0)+L_\Gamma(a^{-1}_2 \widetilde{c}_{p_1}(a_1))+L_\Gamma(b^{-1}_2 \widetilde{c}_{p_1}(b_1))\right)$$
Observe that the parenthesis in the right-hand side of this inequality is smaller than the $\Gamma$-length of the segment of ${\bf p} {\bf q}^{-1}$ between the same elements.
We iterate the argument, starting now with $b^{-1}_2 a_2$ instead of $g$, by considering successively $p_2,p_3,\cdots, p_l$. Since the $\Gamma$-lengths of the various parts sum up, the above observation on the parenthesis in the right-hand side of the inequality directly gives 
$$L_\Gamma(g) \leq C L_\Gamma({\bf p} {\bf q}^{-1}) \leq C P(L_\Gamma(r s^{-1}))$$
in the case where we end with $p_l$ (that is $p_{l+1}$ is trivial in the decomposition given at the beginning of the proof). In the case where we end with a non-trivial $p_{l+1}$, the seemingly loose polynomial distortion applied for the last piece leads to the same conclusion.
The proof in the case of seemingly at most polynomial distortion follows the same scheme.
\end{proof}
The following definition comes from the geometry of relative hyperbolicity, and will be used to show that mixed manifolds are given by undistorted graphs of groups in Lemma \ref{tautologie} and Proposition \ref{mixedprop}. 
\begin{definition}
\label{alls}
Let $G$ be a group with finite generating $S$ and length function $L_S$, and let ${\mathcal H} = \{H_1,\cdots,H_l\}$ be a finite family of finitely generated subgroups of $G$. We say that the family ${\mathcal H}$ {\em has at least linear separation} if there exists $N \geq 0$ and $C \geq 1$ such that, for any $1 \leq i,j \leq l$, for any $u \in G$, if $d_S(H_i,uH_j) = \min \{L_S(\gamma^{-1}_j \gamma_i) \,|\ \gamma_i \in H_i \mbox{, } \gamma_j \in uH_j\} = L$ then, unless $i=j$ and $u \in H_i=H_j$, for any $\gamma_i \in H_i$ with $L_S(\gamma_i) \geq N+L$ one has $$d_S(\gamma_i,u H_j) \geq \frac{1}{C} L_S(\gamma_i).$$
\end{definition}
Notice that a family $\mathcal H$ with at least linear separation is an {\em almost malnormal family of subgroups}, meaning that for any $1 \leq i,j \leq l$, for any $g \in G$, the cardinality of  $H_i \cap g^{-1} H_j g$ is finite unless $i=j$ and $g \in H_j=H_i$. Definition \ref{alls} gives rise to the following adaptation in the setting of graph of groups:
\begin{definition}
\label{ggalls}
Let $\mathcal G = (\Gamma,\{G_e\},\{G_v\},\{\imath_e\})$ be a graph of groups equipped with a set of standard length functions, and let ${\mathcal T}_{\mathcal G} = (\widetilde{\Gamma},\varphi)$ be its universal covering. We say that $\mathcal G$ {\em satisfies the at least linear separation property} if there exists $N \geq 1$ such that any reduced edge-path in $\widetilde{\Gamma}$ with edge-length greater or equal to $N$ contains two consecutive edges $e$ and $f$ such that $\{\imath_e(G_e),\imath_{\bar{f}}(G_f)\}$ forms a family of subgroups of $G_{t(e)} = G_{i(f)}$ with at least linear separation.
\end{definition}
\begin{lemma}
\label{lspd}
A graph of groups with seemingly loose polynomial distortion and at least linear separation property, has loose polynomial distortion (with same degree). Similarly, seemingly at most polynomial distortion and at least linear separation property imply that the graph of groups is undistorted.
\end{lemma}
\begin{proof}
As in Remark \ref{defpratique}, we consider an element $g$ of some vertex-group $G_v$ and two reduced ${\mathcal G}$-sequences $p,q$ with edge-lengths $n$ such that $g = [p q^{-1}]$. We consider their lifts to ${\mathcal T}_{\mathcal G}$, still denoted by $p$ and $q$, we may assume that $p$ starts at $a_1 = 1_{G_v}$ and $q$ at $b_1 = g$. Let us assume that $n \geq N_{\ref{ggalls}}$, the constant given by Definition \ref{ggalls}. Let $p_1$ and $q_1$ be the minimal subsequences of $p$ and $q$ ending at $t(e)$, where $e$ is the edge of ${\mathcal T}_{\mathcal G}$ given by the same Definition \ref{ggalls}. By this definition, the edge-length of $p_1$ is equal to the one of $q_1$ and bounded above by $N_{\ref{ggalls}}$. Still with the notations of this definition, it follows from Definition \ref{alls} that if $c_1$ and $d_1$ denote the elements at which end respectively $p_1$ and $q_1$ then at least one of the two does not belong to $\widetilde{\imath}_{f^{-1}}(G_f)$. Let us now define
$$L=d(\widetilde{\imath}_{e}(G_e),\widetilde{\imath}_{\bar{f}}(G_f))$$ 
and denote by $u_1,v_1$ the elements in $G_{t(e)}$ respectively in $p$ and $q$ starting at $c_1$ and $d_1$. Then, because of seemingly polynomial distortion, there is a polynomial $P$ such that
\begin{enumerate}
  \item If $L_\Gamma(d^{-1}_1c_1) < N_{\ref{alls}}$, then
  $$L_\Gamma(g) \leq  P(L_\Gamma(p_1) +L_\Gamma(q_1)+N_{\ref{alls}}) \leq  P(L_\Gamma(pq^{-1})+N_{\ref{alls}}).$$
\item If $N_{\ref{alls}} + L > L_\Gamma(d^{-1}_1c_1) \geq N_{\ref{alls}}$: by definition of $L$, $L_\Gamma(u_1)+L_\Gamma(v_1) \geq L$. On another hand, similarly to the previous case we have the inequality
\begin{eqnarray*}L_\Gamma(g) &\leq & P(L_\Gamma(p_1) +L_\Gamma(q_1)+L+N_{\ref{alls}})\\
&\leq& P(L_\Gamma(p_1) +L_\Gamma(q_1)+L_\Gamma(u_1)+L_\Gamma(v_1)+N_{\ref{alls}})\\
&\leq & P(L_\Gamma(pq^{-1})+2N_{\ref{alls}}+N_{\ref{alls}}).\end{eqnarray*}
\item If $L_\Gamma(d^{-1}_1c_1) \geq L+N_{\ref{alls}}$, according to Definition \ref{alls} there is a constant $C \geq 1$ such that $L_\Gamma(u_1)+L_\Gamma(v_1) \geq \frac{1}{C} L_\Gamma(d^{-1}_1c_1)$. Hence,
\begin{eqnarray*}L_\Gamma(g) &\leq &  P(L_\Gamma(p_1) +L_\Gamma(q_1)+L_\Gamma(d^{-1}_1c_1))\\
&\leq & P(L_\Gamma(p_1) +L_\Gamma(q_1)+C(L_\Gamma(u_1)+L_\Gamma(v_1)))\\
& \leq & CP( (L_\Gamma(p_1) +L_\Gamma(q_1)+L_\Gamma(u_1)+L_\Gamma(v_1))\leq CP L_\Gamma(pq^{-1}).\end{eqnarray*}
\end{enumerate}
Since each one of the three cases gives a polynomial distortion of the same degree, and they are exhaustive, the proof of the first assertion is complete. The second assertion follows the same scheme: the crossing edge maps are quasi isometric embeddings and at most $N_{\ref{ggalls}}$ of them are composed before meeting a vertex where the linear separation is satisfied.
\end{proof}
\section{Three dimensional manifolds}\label{3dmflds}
In what follows, and in order to set aside some trivial cases, manifolds considered are assumed to be genuine.
\begin{definition}
\label{genuine}
A {\em genuine $3$-manifold} is a $3$-manifold $M$ with the following properties:
\begin{enumerate}
    \item $M$ is compact, connected, irreducible, orientable and with infinite fundamental group.
    \item $M$ is not finitely covered by a torus bundle.
    \item $M$ is {\em $\partial$-irreducible}: if $M$ has a non-empty boundary $\partial M$, then $\partial M$ is a union of tori such that the embedding of each of these tori into the manifold is $\pi_1$-injective.
\end{enumerate}
\end{definition}
Genuine manifolds cannot have Sol geometry as it is virtually a torus bundle over $\ms^1$.
\subsection{Seifert manifolds}\label{sectionseifert}
We refer the reader to \cite{JacoShalen} for instance, among others, for background about Seifert manifolds. 
\begin{definition}
\label{Seifert}
A {\em fibred solid torus of type $(p,q)$} is the suspension 
$$\mt^2_{p,q}=D^2 \times [0,1] / (x,1) \sim(r(x),0)$$ 
of a rotation $r$ of the disc $D^2$ centered at the origin and of angle $\frac{2 \pi p}{q}$. The type, or the fibred torus, is {\em trivial} if the rotation $r$ is the identity. The {\em fibers} of $\mt^2_{p,q}$ are the orbits of the rotation $r$.
\begin{enumerate}
  \item A {\em Seifert manifold} is a genuine $3$-manifold $M$ which is a union of disjoint circles $C_\alpha$,
called the {\em fibers} of $M$, such that each $C_\alpha$ admits a
closed neighborhood $T(C_\alpha)$ which is a union of fibers, and homeomorphic by a fiber-preserving
homeomorphism $h_\alpha$ to some fibred solid torus, whose type gives the {\em type of the fiber}.
  \item A fiber $C_\alpha$ in a Seifert manifold is {\em regular} if the homeomorphism $h_\alpha$ above carries $T(C_\alpha)$ to the trivial fibred torus.
    \item If $M$ is a Seifert manifold, the map $\pi_{M} \colon M\rightarrow B$ which identifies each fiber to a point is called a {\em Seifert fibration} with {\em base} $B$.
\end{enumerate}
\end{definition}
Equivalently, a genuine $3$-manifold is a Seifert manifold if and only if it admits a foliation by circles, see \cite{Epstein}. We now gather some classical results about Seifert manifolds. 
\begin{proposition}
\label{classic Seifert}
Let $M$ be a Seifert manifold.
\begin{enumerate}
  \item The base $B$ of the Seifert fibration of $M$ is a 2-dimensional orbifold with hyperbolic orbifold fundamental group $\pi_1^o(B)$.
  \item All the regular fibers are homotopic. The element of the fundamental group of $M$ defined by a regular fiber is contained in the center of the group. Moreover, if the base is orientable, then it is equal to the center and the fundamental group of $M$ is a central extension:
  $$1 \rightarrow \mz \rightarrow \pi_1(M) \rightarrow \pi^o_1(B) \rightarrow 1$$
  \item If the base of the Seifert fibration is non-orientable then there is a double-cover of $M$ which admits a Seifert fibration with orientable base.
 \item If $M$ has non-empty boundary, then each boundary component is a union of regular fibers, so $\partial M$ is a union of tori.
  \end{enumerate}
\end{proposition}
\begin{proof}
For items (1) and (2), see \cite{Hempel}[Theorems 12.1 and 12.2] for instance: the definition of genuine manifold rules out the few cases where the base has not hyperbolic fundamental group (as an orbifold). Item $(3)$ relies on \cite{Neumann}[Proposition 3.1] which, in the case where $\pi_{M} \colon M \rightarrow B$ is a Seifert fibration with non-orientable base $B$, gives the existence of a Seifert fibration $\pi_{\overline{M}} \colon \overline{M} \rightarrow \overline{B}$ with orientable base $\overline{B}$ and a continuous map $\overline{p} \colon \overline{M} \rightarrow M$ such that the following diagram commutes:
$$\begin{CD}
\overline{M} @>\overline{p}>> M\\
@VV\pi_{\overline{M}}V @VV\pi_{M}V\\
\overline{B} @>p>> B
\end{CD}$$
where $p \colon \overline{B} \rightarrow B$ is the orientable double-cover of $B$. In particular, the pre-image of any fiber $C_\alpha$ of $M$ under the map $\overline{p}$ consists of two fibers of the same type as $C_\alpha$, so that $\overline{p}$ is a local homeomorphism and $\overline{M}$ is a double-cover of $M$. Item (4) is a classic, straightforward consequence of the definition of a Seifert manifold.
\end{proof}
According to Noskov \cite{Nos}, a central extension of any word hyperbolic group has Rapid Decay property, hence Item (1) of Proposition \ref{classic Seifert} above we conclude the following
\begin{corollary}\label{SeifertRD}
The fundamental group of any Seifert manifold has the Rapid Decay property.
\end{corollary}
A {\em boundary subgroup} is a subgroup of the fundamental group of a manifold $M$ defined by some boundary component of $\partial M$.  An {\em incompressible surface $S$ in $M$} is an embedded surface $i \colon S \rightarrow M$ such
that $i_{\#} \colon \pi_1(S) \rightarrow \pi_1(M)$ is
injective and $i(S)$ is not parallel to a boundary component, i.e. cannot be isotoped into the boundary of $M$. 
\begin{lemma}
\label{preparation}
Let $M$ be a Seifert manifold with non-empty boundary $\partial M$. 
\begin{enumerate}
\item $M$ is finitely covered by a trivial surface bundle over $\ms^1$.
\item For any two boundary subgroups $H_1, H_2$, and any element $g \in \pi_1(M)$, then
$$g^{-1} H_1 g \cap H_2 = \langle z \rangle,$$
if $z \in \pi_1(M)$ is the element defined by the regular fiber, unless $H_1=H_2$ and $g\in H_1=H_2=\mz^2$. Moreover, the boundary subgroups are quasi isometrically embedded in the fundamental group of the manifold.
\end{enumerate}
\end{lemma}

\begin{proof}
\underline{Item $(1)$}: This is a classical result, we summarize here the idea in \cite{Jaco}[VI.26, page 103] (see also \cite{Gautero1} for another point of view): since $M$ has non-empty boundary and is $\partial$-irreducible, by Poincar\'e duality the non-trivial element $z\in H^1(M;\mz)$ defined by a regular fiber has infinite order and is dual to a non-trivial element in $H_2(M,\partial M;\mz)$ (see \cite{Thurstonnorm}[Lemma 1]); hence all the regular fibers are transverse to, and intersect positively, a same $2$-sided, non-separating incompressible surface $S$ embedded in $M$, which is a representative of the image of $z\in H^1(M;\mz)$ in the dual $H_2(M,\partial M;\mz)$; since the regular fibers are positive powers of each exceptional fiber, these last ones also intersect positively $S$; it follows that $S$ is a cross-section to a flow whose all orbits are closed since these are the fibers of the Seifert fibration. Hence $M$ is a mapping-torus of a finite-order homeomorphism of $S$. 

\underline{Item $(2)$}: Take $T_1,T_2\subseteq\partial M$ two boundary tori such that $H_i=\pi_1(T_i)$ for $i=1,2$. Since the boundary of a Seifert manifold is a union of regular fibers, by Proposition \ref{classic Seifert} (2) the fundamental group of the regular fiber is in the center of the fundamental group $\pi_1(M)$, and hence we have the inclusion $\langle z \rangle \subseteq g^{-1} H_1 g \cap H_2$. If there is another element in  $g^{-1} H_1 g \cap H_2$ then this intersection contains a power of the element $\gamma$ represented by some boundary component in $\partial S \cap T_2$: indeed each boundary subgroup of $M$ is generated by the regular fiber and an element in the torus which intersects it exactly once (a meridian); thus, if $g^{-1} H_1 g \cap H_2$ contains an element distinct from the regular fiber then it contains a power of this meridian and therefore a power of $\gamma$ since $H_2$ is abelian. Since $S$ is the fiber of the fibration given by Item $(1)$, the subgroup $N$ that it defines is normal in $\pi_1(M)$. It follows that $g \gamma^l g^{-1} \in H_1 \cap N$. Since $\gamma^l$ is a power of the element defined by a boundary circle of $S$ then so is the element $g \gamma^l g^{-1}$. The only possibility is two boundary circles in $\partial S$ being homotopic. This implies that they either belong to the same boundary torus of $M$, or $S$ is an annulus, that is $M$ is the mapping-torus of a homeomorphism of the annulus. In the former case, they are homotopic in the boundary torus, which does not agree with $g \notin H$. The latter case in turn implies that $M$ is finitely covered by a torus bundle over the interval: this possibility has been ruled out in the definition of a genuine $3$-manifold, hence the conclusion.
\end{proof}
\subsection{Atoroidal $3$-manifolds}
\label{sectionatoroidal}
A compact $3$-manifold is {\em atoroidal} if it does not contain any
incompressible torus non isotopic to a boundary torus (if any). The Geometrization Theorem (\cite{Perelman} - see \cite{BBMBP} for a complete text) classifies the possible geometries, hence the fundamental groups.
\begin{proposition}
\label{evidence} The fundamental group of any genuine
atoroidal $3$-manifold has the Rapid Decay property.
\end{proposition}
\begin{proof}
By the Geometrization Theorem, a closed orientable atoroidal irreducible $3$-manifold or the interior of an orientable atoroidal irreducible $3$-manifold which has tori in its boundary, admits one of the eight following geometries, namely:
$$\me^3,\ \mh^3,\ \ms^3,\ \ms^2 \times \mr,\ \mh^2 \times \mr,\ \widetilde{SL_2(\mr)},\ \mathrm{Nil}\hbox{ and }\mathrm{Sol}$$ 
For the $\mh^3$-geometry, i.e. hyperbolic manifolds, a result of Jolissaint \cite{Jolissaint}[Theorem 3.2.1] gives the Rapid Decay property for the
fundamental group of closed such manifolds. In the case of tori in the boundary, the Rapid Decay property comes from \cite{Farb} (the fundamental group is strongly hyperbolic relative to the boundary subgroups) and \cite{indira} (a group which is strongly hyperbolic relative to polynomial growth subgroups has the Rapid Decay property). By \cite{PeterScott}, when genuine the six geometries
$\me^3, \ms^3, \ms^2 \times \mr, \mh^2 \times \mr,
\widetilde{SL_2\mr}$ and $\mathrm{Nil}$ are Seifert
manifolds, treated in Corollary \ref{SeifertRD}.
\end{proof}
\subsection{Graph manifolds}
\label{sectiongraph}
Our next class of 3-dimensional manifolds is given by gluing together Seifert manifolds along incompressible tori, and the goal of this subsection is to prove Proposition \ref{GraphMfldPolDisto}.
\begin{definition} A {\em graph manifold} is a genuine, non Seifert $3$-manifold
$M$ which admits a finite, non-empty collection of incompressible tori
$T_1,\cdots,T_r$ such that the closure of each connected component
of $M\setminus \bigcup^r_{i=1} T_i$ is a Seifert manifold.
\end{definition}
\begin{remark}
The fundamental group of a graph manifold is the fundamental group of a graph of groups $\mathcal G =(\Gamma,\{G_e\},\{G_v\},\{\imath_e\})$ whose vertex-groups are fundamental groups of Seifert manifolds with boundary, edge-groups are $\pi_1(T_i)=\mz \oplus \mz$ subgroups and each morphism $\imath_e$ from an edge-group into a vertex-group is induced by a homemorphism of the torus, with image some boundary component of the associated Seifert manifold. The requirement for a Seifert manifold to be genuine (see Definition \ref{genuine}) forbids (finite covers of) torus bundles in the decomposition of a graph manifold into Seifert manifolds. This is not a restriction because a graph manifold is also genuine, hence is not a torus bundle over $\ms^1$. Thus such a component would be the mapping-torus of a homeomorphism of the annulus, which can be omitted by changing the gluing homeomorphism of its boundary components to the other Seifert components.
\end{remark}
To show that a graph of groups describing a graph manifold has at most polynomial distortion, which is the content of Proposition \ref{GraphMfldPolDisto}, the strategy is to prove seemingly polynomial distortion, as well as tight dynamics, and conclude using Lemma \ref{spdtdpd}.
\begin{lemma}
\label{intermediaire1}
Let $G=\pi_1(M)$, where $M$ is a graph manifold, and let $T_1,\cdots,T_r$ be a family of incompressible tori such that the closure of each connected component $S_j$ of $M\setminus \bigcup^r_{i=1}T_i$ is a Seifert manifold. Thus $G=\pi_1({\mathcal G})$ where $\mathcal G =  (\Gamma,\{G_e\},\{G_v\},\{\imath_e\})$ is a graph of groups whose vertex-groups are fundamental groups of Seifert manifolds with boundary, edge-groups are $\pi_1(T_i)=\mz \oplus \mz$ subgroups and each morphism $\imath_e$ from an edge-group into a vertex-group is induced by a homemorphism of the torus, with image some boundary component of the associated Seifert manifold.
\begin{enumerate}
    \item $\mathcal G$ seemingly has at most polynomial distortion (Definition \ref{defspd}), in fact is seemingly undistorted.
    \item $\mathcal G$ has tight dynamics (Definition \ref{ggtd}).
\end{enumerate}
\end{lemma}
We first need to understand the behavior of the gluing homeomorphisms. Recall that an orientation-preserving diffeomorphism $\phi \colon \mt^2 \rightarrow \mt^2$ is called {\em Anosov} if there is $\lambda> 1$ and an eigenspace decomposition of the tangent bundle $T \mt^2 = T^u \mt^2 \oplus T^s \mt^2$ such that at each point $x \in \mt^2$,one has $||D_x\phi(v)|| = \lambda
||v||$ (resp. $||D_x\phi(v)|| = \frac{1}{\lambda} ||v||$) for all $v
\in T^u_x \mt^2$ (resp. for all $v \in T^s_x \mt^2$).  From the Nielsen-Thurston classification (see \cite{FLP} for instance), any orientation-preserving homeomorphism of the torus $\mt^2$ which is not isotopic to a periodic homeomorphism, nor to an Anosov homeomorphism is isotopic to a reducible homeomorphim which fixes a simple closed curve $C$ and acts as a twist on the annulus $\mt^2 \setminus \{C\}$. Since we are only interested in the action of the gluing homeomorphisms on the fundamental group of the torus, we may thus assume that they are either periodic (i.e. finite-order), Anosov, or reducible. Moreover, from \cite{Franks} and \cite{Newhouse}, any Anosov diffeomorphism of the torus is conjugate to a linear one, that is defined by a linear map of $\mr^2$ whose matrix is in $SL_2(\mz)$ with no eigenvalues of modulus $1$. We thus only have to deal with linear Anosov maps.
 \begin{definition}
 Let $\phi$ be a linear Anosov diffeomorphism, and denote by $v_u, v_s$ the unit eigenvectors associated to the eigenvalues $\lambda_{\phi} > 1$ and $\frac{1}{\lambda_{\phi}} < 1$. Any element $\gamma$ of the fundamental group of the torus corresponds to a vector $v_\gamma$ (with integer coefficients in the canonical basis). In the basis given by $(v_u, v_s)$ we have $v_\gamma = a_u(\gamma) v_u + a_s(\gamma) v_s$. The {\em $\phi$-foliation length of $\gamma$}, is defined by 
 $$|\gamma|_\phi = |a_u(\gamma)| + |a_s(\gamma)|$$
 and coincides with the length of $\gamma$ in the generating set $\{v_u, v_s\}$. The {\em $\phi$-slope of $\gamma$}, is given by 
 $$sl_\phi(\gamma) = \displaystyle \frac{|a_u(\gamma)|}{|a_s(\gamma)|}.$$
 \end{definition}
The following lemma concerning linear Anosov diffeomorphisms is classic, but we include it here for completeness because we couldn't find a reference.
\begin{lemma}
\label{Anosov}
Let $\phi \colon \mt^2 \rightarrow \mt^2$ be a linear Anosov diffeomorphism of the torus. For any $\gamma \in \pi_1(\mt^2)=\left<S\right>$, with $S$ a fixed finite generating set, there exists $M_\gamma \geq 0$ such that for any $n \in \mz$:
\begin{enumerate}
  \item $\min_{j \in \mz} \{L_S(\phi^j(\gamma^n))\}$ is attained for $-M_\gamma \leq j \leq M_\gamma$ and $L_S(\phi^j(\gamma^n)) > L_S(\gamma^n)$ for $|j| > M_\gamma$. In particular, no element $\gamma$ is sent by $\phi$ to a power of itself. Moreover $L_S(\phi^j(\gamma^n))$ grows exponentially in $n$, for any $|j| \geq M_\gamma$.
  \item If $\eta\in \pi_1(\mt^2)$ is such that $\{\gamma,\eta\}$ generate $\pi_1(\mt^2)$ then there is $C_{\gamma,\eta} \geq 1$ such that, if $\phi^j(\gamma^n) = \gamma^{un} \eta^{vn}$ with $L_S(\phi^j(\gamma^n)) < L_S(\gamma^n)$ then $L_S(\eta^{vn}) \geq \frac{1}{C_{\gamma,\eta}} L_S(\gamma^n)$.
\end{enumerate}
\end{lemma}
\begin{proof}
Since $S$ is finite, the word metric is quasi-isometric to the metric given by the $\phi$-foliation length, so it is enough to prove the lemma for this $\phi$-foliation length defined above.

\underline{Item (1)}: It suffices to do the case $n=1$.  The general case follows from the equality $sl_\phi(\gamma) = sl_\phi(\gamma^n)$ for any $n \in \mz$, which is deduced straight from the definition of $sl_\phi(\gamma)$. A simple calculation gives the equivalence: 
$$|\phi(\gamma)|_\phi < |\gamma|_\phi \Leftrightarrow sl_\phi(\gamma) < \lambda_{\phi},$$ 
showing that the slope is bounded, and the equality 
$$sl_\phi(\phi(\gamma)) = \lambda^2_\phi sl_\phi(\gamma),$$ 
implying that the slope grows applying $\phi$. Of course the analogous relations hold when applying $\phi^{-1}$, namely $|\phi^{-1}(\gamma)|_\phi < |\gamma|_\phi \Leftrightarrow sl_\phi(\gamma) > \frac{1}{\lambda_{\phi}}$, and hence $sl_\phi(\phi^{-1}(\gamma)) = \frac{1}{\lambda^2_\phi} sl_\phi(\gamma)$.  In both cases, since $\lambda_\phi > 1$ the minimum is attained after a number of iterations which only depends on $sl_\phi(\gamma)$ and $\lambda_\phi$.

\underline{Item $(2)$}: We first consider the $j^{th}$ power at which the minimum given by Item $(1)$ is attained. Then $sl_\phi(\gamma^{un} \eta^{vn})$ is close (that is up to a constant depending only on $\phi,\gamma$) to $1$. Since $sl_\phi(\gamma^{un}) = sl_\phi(\gamma)$, a simple computation gives a constant $D \geq 1$, depending on $\gamma$ and $\eta$, such that the power of $\eta$ at this minimum is at least $D$ times the power of $\gamma$. Since $-M_\gamma \leq j \leq M_\gamma$ and $j \neq 0$, it follows the existence of such a constant $D$ for each one of these powers. We so get the existence of a constant $C_{\gamma,\eta} \geq 1$ such that for any integer $n$, the length of the power of $\eta$ in $\phi^j(\gamma^n)$ is at least $\displaystyle \frac{1}{C_{\gamma,\eta}}$ times the length of $\gamma^n$.
\end{proof}
We can now prove Lemma \ref{intermediaire1}
\begin{proof}[Proof of Lemma \ref{intermediaire1}]
\underline{\it Part (1): Seemingly polynomial distortion}. Since the boundary subgroups of the Seifert manifolds are quasi isometrically embedded, the crossing edge maps are quasi isometric embedding. If the edge-length of a reduced edge-path is strictly greater than the diameter $D$ of the maximal tree in $\Gamma$, this edge-path contains a loop. If all reduced edge-paths $p$ such that there exist vertex-elements $g$ and $h$ with $[g] = [p^{-1} h p]$ as in Definition \ref{defspd} have edge-length smaller than $D+1$ then $\mathcal G$ seemingly has at most polynomial distortion, in fact is undistorted, since a finite composition of q.i. embeddings is a q.i. embedding. Without loss of generality, it is thus sufficient to consider a simple loop $p=e_1 \cdots e_l$ in $\Gamma$. It follows from Lemma \ref{preparation} (2) that, to prove the seemingly at most polynomial distortion, we only have to study what happens to (the powers of the) element defined by the regular fibers of the Seifert manifolds associated to the vertex-groups when ``passing through $p$''. To be more precise, let us choose arbitrarily one of them, say $G_v$ with $v =i(p)$, and let us denote by $\gamma$ this element. We denote by $\eta_l$ the element defined by the meridian of the boundary component corresponding to the subgroup $\imath_{e_l}(G_{e_l})$. We want to understand the form of $c_p(\gamma)=c_{e_l} \circ c_{e_{l-1}} \circ \cdots \circ c_{e_1}(\gamma)$. By the Nielsen-Thurston classification and Lemma \ref{Anosov}, the three possibilities are (we recall that each fiber of each boundary component is a regular fiber - see Lemma \ref{classic Seifert}):
\begin{enumerate}
\item[Case 1:] $c_p(\gamma^n) = \gamma^{\pm n}$ and the length of $\gamma$ is preserved.
\item[Case 2:] $c_p(\gamma^n) = \gamma^{\pm n} \eta^{nm}_l$ for some non-zero integer $m$.
\item[Case 3:] $c_p(\gamma^n) = \gamma^{nj} \eta^{nm}_l$ for some integer $j$ with $|j| \neq 1$ and non-zero integer $m$.

\end{enumerate}

In case 1 the length is preserved so no distortion. In cases $2$ and $3$, there is no edge-path $q$ containing $p$ for which there exist elements $g$ and $h$ with $[g] = [q^{-1} h q]$. It follows that $\mathcal G$ seemingly has at most polynomial distortion, in fact is seemingly undistorted.

\bigskip

\underline{\it Part (2): Tight dynamics}.
We consider an edge $e$ with vertices $v=i(e)$ and $w=t(e)$, we need to understand the crossing map $c_e$ on $\gamma_v$ and $\gamma_w$ the regular fibers of the Seifert manifolds associated to $G_v$ and $G_w$ and $\eta$ the meridian of $G_w$.

If $c_e(\gamma^n_v)$ satisfies Cases $1$ or $2$ above, then the first condition of Definition \ref{ggtd} is satisfied.

Let us thus assume that $c_e(\gamma^n_v)$ satisfies Case $3$. If $j \neq 0$, $c_e$ is induced by a pseudo-Anovov homeomorphism. Assume first that $L_{G_w}(c_e(\gamma^n_v)) \geq L_{G_v}(\gamma^n_v)$ holds. Then either the first condition of Definition \ref{ggtd} is satisfied or the exponential dilatation given by Item $(1)$ of Lemma \ref{Anosov} gives us a constant $C \geq 1$, depending only on $\mathcal G$, such that the length of $\eta^{mn}$ in $c_e(\gamma^n_v)$ is at least $\frac{1}{C}$ times the length of $\gamma^n_v$.

If $c_e(\gamma^n_v)$ satisfies Case $3$, we therefore only have to consider the case where $L_{G_w}(c_e(\gamma^n_v)) < L_{G_v}(\gamma^n_v)$ holds. Item $(3)$ of Lemma \ref{Anosov} gives in this case a constant $C \geq 1$, depending only on $\mathcal G$, such that the length of $\eta^{mn}$ in $c_e(\gamma^n_v)$ is at least $\frac{1}{C}$ times the length of $\gamma^n_v$.

We still have to prove that there is at least ``linear divergence'' from the other boundary subgroups, and their left-classes, when the power of $\eta$ increases in order to get the constant $C$ required in the second alternative of Definition \ref{ggtd}. By Item $(1)$ of Lemma \ref{preparation}, the Seifert manifold fibers over the circle with fiber a surface $S$. The geodesics associated to elements defined by the boundary circles in $\partial S$ of two distinct boundary components of the manifold, or a boundary circle and a non-trivial conjugate of it, diverge exponentially in the surface. Since the monodromy is of finite order (Item $(1)$ of Lemma \ref{preparation}), the same assertion holds in the $3$-manifold. From Item $(2)$ of Lemma \ref{preparation}, this exponential divergence also holds between the cyclic subgroup $\langle \eta \rangle$ and any (left-class of a) different boundary subgroup, or any different left-class of the same boundary subgroup. Hence the existence of a constant $C \geq 1$ as announced. The case $j=0$ in Case $3$ is now straightforward from which precedes. We so proved that $\mathcal G$ has tight dynamics.
\end{proof}
\begin{proof}[Proof of Proposition \ref{GraphMfldPolDisto}] Combine the above result with Lemma \ref{spdtdpd}.
\end{proof}
\subsection{Mixed $3$-manifolds}\label{sectionmixed} We borrow the terminology of mixed $3$-manifold from \cite{PrWise} and call {\em mixed $3$-manifold} any genuine non-atoroidal $3$-manifold which is neither a Seifert nor a graph manifold. The structure of mixed manifolds is well-understood since the celebrated {\em JSJ-decomposition}, that we now recall. 
\begin{theorem}[JSJ-decomposition, \cite{JacoShalen}]
\label{JSJD}
Given any mixed $3$-manifold $M$, there exist a possibly empty family of maximal
Seifert or graph submanifolds $J_1,\cdots,J_r$ in $M$ and of incompressible tori $T_1,\cdots,T_l$ not belonging to these former submanifolds
such that the closure of each connected component of $M \setminus
(\bigcup^r_{i=1} J_i \cup \bigcup^l_{i=1} T_i)$ is a compact $3$-manifold with hyperbolic interior, the boundary of which is a union of incompressible tori.
\end{theorem}
The JSJ-decomposition gives us the following graph of groups decomposition of the fundamental group of a mixed manifold.
\begin{lemma}
\label{tautologie}
The fundamental group of a mixed $3$-manifold $M$ is the fundamental group of a graph of groups ${\mathcal G}$ with the following properties:
\begin{enumerate}
  \item The vertex-groups are fundamental groups of genuine $3$-manifolds with boundary tori which have either hyperbolic interior, in which case are termed {\em hyperbolic vertices}, or are Seifert or graph manifolds. The edge-groups are $\mz \oplus \mz$-subgroups. The edge-morphisms are induced by tori-homeomorphisms and their images are the boundary subgroups of the vertex-groups, into which they are quasi isometrically embedded.
  \item For any edge $e$, if some vertex-group of $e$ is the fundamental group of a Seifert or graph manifold, then the other one is a hyperbolic vertex.
\item The graph of groups ${\mathcal G}$ satisfies the at least linear separation property and seemingly has at most polynomial distortion, in fact is seemingly undistorted.

\end{enumerate}
\end{lemma}
\begin{proof}
\underline{\it Part (1).} Consider any Seifert or graph submanifold $M_1$ in the decomposition of $M$ given by Theorem \ref{JSJD}. It is not separated from another Seifert or graph manifold $M_2$ by an incompressible torus because otherwise $M_1 \cup M_2$ is a graph submanifold of $M$ containing both $M_1$ and $M_2$ which is a contradiction with the maximality of the submanifolds in the JSJ-decomposition of Theorem \ref{JSJD}. 

\medskip

\underline{\it Part (2).} Any Seifert or graph submanifold in this decomposition is glued along its boundary tori to manifolds with hyperbolic interior. Since the incompressible tori given by Theorem \ref{JSJD} might only belong to submanifolds with hyperbolic interior, the conclusion follows. 

\medskip

\underline{\it Part (3).} The collection of boundary subgroups of the fundamental group of a finite volume manifold with hyperbolic interior $M$ is an almost malnormal collection ${\mathcal H}$ of $\mz^2$ coming from the incompressible tori $T_1,\cdots,T_l$ satisfying the following (also called Bounded Coset Penetration Property, see \cite{Farb} - this is a key-point in proving that the fundamental group of the manifold being strongly relatively hyperbolic with respect to the boundary subgroups):

There are constants $C,D > 0$ such that, for any two distinct boundary subgroups $H_0, H_1\in {\mathcal H}$ and elements $l_i \in H_i$ ($i=0,1$), geodesics representing $l_0$ and $l_1$ are $D$-close one to each other along a length of at most $C$ and outside this region diverge exponentially. This also holds for any loop in $H_i$ and a conjugate by an element $g \notin H_i$. This is a consequence of the malnormality and the thinness of the geodesic rectangles not contained in the boundary subgroups.


This bounded coset penetration property implies the at least linear separation property of Definition \ref{alls} for the family formed by the embeddings of the edge-subgroups into any hyperbolic vertex-group. Part (2) of this lemma implies that in any $\mathcal G$-sequence any vertex-group which is not hyperbolic is preceded and followed by a hyperbolic vertex-group. Hence $\mathcal G$ satisfies the at least linear separation property of Definition \ref{ggalls}. Since the embeddings of the boundary tori in graph, Seifert or hyperbolic manifolds are quasi isometric embeddings, this same argument gives the seemingly at most polynomial distortion, in fact seemingly undistortion since a composition of q.i. embeddings is a q.i. embedding.
\end{proof}
Lemma \ref{tautologie} gives readily gives the following

\begin{proposition}
\label{mixedprop}
The graph of groups $\mathcal G$ is undistorted.
\end{proposition}

\begin{proof}
This is implied by $(3)$ of Lemma \ref{tautologie} and Lemma \ref{lspd}.
\end{proof}

\begin{corollary}\label{mixed}
Fundamental groups of mixed $3$-manifolds have Rapid Decay property.
\end{corollary}
\begin{proof}
This comes from Proposition \ref{distortion}, combined with Proposition \ref{mixedprop} with the fact that all the vertex subgroups have Rapid Decay property (Propositions \ref{evidence} and \ref{GraphMfldPolDisto}).
\end{proof}

\bibliographystyle{plain}
\bibliography{biblioRD}

\end{document}